\documentclass[a4paper,11pt,twoside,reqno]{amsart} 
\DeclareSymbolFontAlphabet{\mathcal}{symbols}
\usepackage[cmtip,all]{xy}
\numberwithin{equation}{section}
\usepackage{amsmath}
\usepackage[psamsfonts]{amssymb}
\usepackage[psamsfonts]{eucal}
\usepackage{amsthm}
\usepackage[psamsfonts]{amssymb}
\usepackage{mathrsfs}
\usepackage{framed}
\usepackage{hyperref}
\hypersetup{
    colorlinks,
    citecolor=blue,
    filecolor=black,
    linkcolor=blue,
    urlcolor=black
}
\usepackage[a4paper]{geometry}

\usepackage{enumerate}
\usepackage{graphicx}
\usepackage{tikz}
\usepackage{color}
\definecolor{trama}{gray}{.875}
\usepackage{lmodern}
\usepackage{textcomp}
\usepackage{url}

\newtheorem{theorem}{Theorem}

\newtheorem{proposition}{Proposition}[section]

\newtheorem{theoremv}[proposition]{Theorem}

\newtheorem{corollaire}[proposition]{Corollary}

\theoremstyle{definition}
\newtheorem{definition}[proposition]{Definition}
\newtheorem{exemple}[proposition]{Example}

\newtheorem{remarque}[proposition]{Remark}

\newcommand{\secref}[1]{Section~\ref{#1}}

\newcommand{\thmref}[1]{Theorem~\ref{#1}}
\newcommand{\propref}[1]{Proposition~\ref{#1}}

\newcommand{\corref}[1]{Corollary~\ref{#1}}

\newcommand{\remref}[1]{Remark~\ref{#1}}
\newcommand{\exemref}[1]{Example~\ref{#1}}

\newcommand{\defref}[1]{Definition~\ref{#1}}

\def\ov{\overline}
\newcommand{\cU}{\mathcal U}
\newcommand{\cF}{\mathcal F}
\def\cAb{{\mathcal Ab}}
\def\gd{{\mathfrak{d}}}
\def\gC{{\mathfrak{C}}}
\def\gH{{\mathfrak{H}}}
\def\gI{{\mathfrak{I}}}
\def\gT{{\mathfrak{T}}}
\def\tc{{\mathtt c}}

\def\tv{{\mathtt v}}
\def\tw{{\mathtt w}}
\def\N{\mathbb{N}}
\def\R{\mathbb{R}}
\def\Z{\mathbb{Z}}
\def\im{{\rm Im\,}}
\def\id{{\rm id}}
\def\codim{{\rm codim\,}}
\def\pr{{\rm pr}}

\def\sing{{\mathrm{sing}}}
\def\sd{{\mathrm{sd}\,}}
\def\depth{{\rm depth} \,}
\def\rc{{\mathring{\tc}}}
\def\cS{{\mathcal S}}
\def\cT{{\mathcal T}}
\def\K{{\rm K}}
   \def\sda{{\mathfrak{sd}\,}}

 \newcommand{\menos}{\backslash} 
    
\newcommand{\bi}[2]{{#1}^{^{#2}}}
\newcommand{\ib}[2]{{#1}_{_{#2}}}

\newcommand{\hiru}[3]{{#1}_{_{#2}}{( #3 )}}
\newcommand{\lau}[4]{{#1}^{^{#2}}_{_{#3}}{( #4 )}}
\newcommand{\bost}[5]{{#1}^{^{#2}} \! {#3}_{_{#4}}{( #5 )}}

\author{David Chataur}
\address{Lafma\\
Universit\'e de Picardie Jules Verne\\
33, rue Saint-Leu\\
80039 Amiens Cedex~1\\
         France}
\email{David.Chataur@u-picardie.fr}

\author{Martintxo Saralegi-Aranguren}
\address{Laboratoire de Math{\'e}matiques de Lens EA 2462 \\
       F\'ed\'eration CNRS Nord-Pas-de-Calais FR2956\\
      Universit\'e d'Artois\\
          62307 Lens Cedex\\
         France}
\email{martin.saraleguiaranguren@univ-artois.fr}

\author{Daniel Tanr\'e}
\address{D\'epartement de Math{\'e}matiques\\
         UMR 8524 \\
         Universit\'e de Lille\\
         59655 Villeneuve d'Ascq Cedex\\
         France}
\email{Daniel.Tanre@univ-lille.fr}

\title[General perversities and topological invariance]{Intersection homology. \\ General perversities and topological invariance}

\date{\today}

\thanks{The third author is partially supported by the MINECO grant  MTM2016-78647-P  and ANR-11-LABX-0007-01  ``CEMPI''}

\keywords{Intersection homology; Topological invariance ; Perversiy; CS set; Intrinsic CS set; Pseudomanifold.}

\subjclass{55N33, 58A35, 57N80}

\begin{document}

\begin{abstract} 
Topological invariance of the intersection homology of a pseudomanifold without codimension one strata, proven by
Goresky and  MacPherson, is one of the main features of this homology.
This  property is true for  codimension-dependent  perversities with some growth conditions, 
verifying $\ov p(1)=\ov p(2)=0$. 
King reproves this invariance by associating an intrinsic pseudomanifold $X^*$ to any  pseudomanifold $X$.
 His proof consists of an isomorphism between the associated intersection homologies  
 $H^{\ov{p}}_{*}(X)  \cong   H^{\ov{p}}_{*}( X^*)$
for any  perversity $\ov{p}$  with the same growth conditions verifying $\ov p(1)\geq 0$.

In this work, we prove a certain topological invariance within the framework of strata-dependent perversities, $\ov{p}$, 
which corresponds to the classical topological invariance if $\ov{p}$ is a GM-perversity.
We also extend it to the tame intersection homology, a variation of the intersection homology, particularly suited for 
``large'' perversities, if there is no singular strata on $X$ becoming  regular  in
$X^*$.
In particular, under the above conditions, the intersection homology  and the tame intersection homology 
are invariant under a refinement of the stratification.
\end{abstract}

\maketitle


The definition of intersection homology  relies on a stratification of the space and it
is natural to ask for its impact. 
In fact, Goresky and MacPherson proved, in their first work about intersection homology (\cite{GM1}),
that there is no dependence  of the stratification, a main property of this theory.
More precisely,  if  $\ov{p}\colon \Z_{>1}\to \Z$ is a classical perversity verifying 
$\ov{p}(t)\leq\ov{p}(t+1)\leq\ov{p}(t)+1$, $\ov{p}(2)=0$, then the $\ov{p}$-intersection homology
is a topological invariant of a pseudomanifold without codimension one strata. 
In the sequel, such perversity is called a GM-perversity.

Intersection homology  also exists  for more general  perversities, 
$\ov{p}\colon \cS_{X}\to \Z$,
defined on the family of strata of a filtered space $X$ and taking values in  $\Z$, cf. \cite[Section 1.1]{MacPherson90}
and \remref{rem:pourquoigenerales}. 
Since these perversities are defined on the set of strata it is natural to conclude that 
we could not expect any topological invariance in such a general context.
In this work, we present some conditions on  perversities and  stratifications which  lead to a topological invariance
with this generality.

To achieve this goal, we first develop the properties of a complex, $C_{*}^{\ov{p}}(X;G)$, giving intersection homology.
Its main feature lies in the choice of the simplices, $\sigma\colon \Delta\to X$, called
\emph{filtered simplices.} They are such that,
for any stratum $S$ of $X$, the intersection $\sigma^{-1}(S)\cap \Delta$ is 
 either the emptyset or a face of $\Delta$. Such complex is already present in
 the Appendix of \cite{CST1}.
Its suitability with the set of  strata, makes it possible to proceed to a blow-up along the strata
and thus to develop a cohomology well adapted to the geometry of the singular spaces.
We study it  in \cite{CST4}, where we prove the existence of cap and cup products for coefficients
in any commutative ring. The existence of a Poincar\'e isomorphism via a cap product, between this 
blown-up cohomology and the intersection homology, is presented in \cite{CST2}, see also \cite{ST2}. 
Let us notice as well that the use of these simplices allows a decomposition of Borel-Moore chains as
a locally finite sum of \emph{intersection chains,} which brings an efficient tool in \cite{CST5} and \cite{ST1}. 
On the theoretical aspect,
we can also note the unification of the occuring objects, the domain of the filtered simplices lying
in the same category as the spaces: they are particular cases of 
locally cone-like spaces introduced by Siebenmann in \cite{Sieb} and called CS sets.

\smallskip
The first point is the proof that the complex built on filtered simplices gives the usual intersection homology.
The case of codimension-dependent perversities has been considered in \cite[Proposition A.29]{CST1}; we 
present here a proof for perversities defined on the set of strata.
Let us denote by $I^{\ov{p}}C_{*}(X;G)$
the complex introduced by  Goresky and MacPherson in \cite{GM1}, 
adapted to general perversities as in \cite{MacPherson90} or \cite{FriedmanBook}.

\begin{theorem} 
Let $X$ be a CS set and $\ov{p}\colon \cS_{X}\to \Z$ be a perversity.
The canonical inclusion 
$C_{*}^{\ov{p}}(X;G)\hookrightarrow I^{\ov{p}}C_{*}(X;G)$
is a chain map inducing an isomorphism in  homology.
\end{theorem}

The proof 
 needs the existence of a Mayer-Vietoris sequence (\propref{prop:MVhomologie}), 
a stratified homotopy invariance (\propref{prop:homotopesethomologie}) and the determination
of the intersection homology of a cone (\propref{prop:homologiecone}).
Under this generality, \thmref{thm:martinmacphersonhomologie} is used in \cite{CST4,CST2}.

\medskip
If one works with perversities taking values outside of the range $[\ov{0},\ov{t}]$ defined by Goresky and MacPherson
(\cite{GM1}),
some fundamental properties disappear. Recall (see \defref{def:allowableMacPherson})
 the key notion of $\ov{p}$-allowable simplices.
In the classical range of \cite{GM1}, 
the support of the boundary  of a  $\ov{p}$-allowable simplex is never included in the singular set.
Outside this range, this property is no longer true, 
with bad consequences on the behaviour of intersection homology.
In \cite{MR2210257} (see also \cite{MR2276609})  a modification of intersection homology is presented, 
matching with the usual intersection homology if $\ov{p}\leq\ov{t}$. The idea is to eliminate the 
$\ov{p}$-allowable simplices 
included in the singular part, see \defref{def:tameGreg}.
 With this modification, 
called \emph{tame intersection homology,}
the deRham isomorphism (\cite{MR2210257}) and the Poincar\'e Duality  (\cite{CST2,FM,ST1,ST2}) 
remain true. 
We present  a filtered version, $\gC _*^{\ov p} (X;G )$, of the tame intersection complex  
and denote $I^{\ov{p}} \gC_* (X;G )$ the non-GM complex introduced by G. Friedman in \cite{FriedmanBook}.

\begin{theorem} 
Let $X$ be a CS set and $\ov{p}\colon \cS_{X}\to \Z$ be a perversity.
The canonical inclusion 
$ \gC _*^{\ov p} (X;G )\hookrightarrow I^{\ov{p}} \gC_* (X;G )$
is a chain map inducing an isomorphism in  homology.
\end{theorem}

After the description and the study of these two complexes, comes the second part of this work,
the research of a topological invariance.
To this end, we consider the method used by H.~King \cite{MR800845} 
in the framework of CS sets.
Following a process  credited to D. Sullivan,
King associates to any CS set, $X$,
a new CS set,  $X^*$, with the same underlying topological space but endowed with a different stratification.
The thrust in the construction of $X^*$ is the elimination of the strata of $X$ that are not  
topologically singular. %
For example, let us consider a sphere having a point $P$ as  singular stratum.
The point $P$ and any other point of the sphere have homeomorphic neighborhoods. 
So, the point $P$ has no  singular 
character relatively to the other points. 
In this case, the associated CS set $X^*$ is the sphere with just one (regular) stratum.

Let us come back to the general case. As two CS sets with the same underlying topological space have the same associated CS set,  we call $X^*$ an  \emph{intrinsic CS set.}
In the framework of a GM-perversity $\ov{p}$, the topological invariance means that 
the $\ov{p}$-intersection homologies of $X$ and  $X^*$ are isomorphic.

Therefore, we   address the issue of topological invariance from the following point of view:
for which general  perversities, do we have an isomorphism between the intersection homologies of $X$ and  $X^*$?
The first challenge arises from the perversities. 
In  the study of topological invariance with GM-perversities, we consider
 the \emph{same} perversity  on $X$ and $X^*$.
 An analogy of this situation is the case of a perversity defined on $X^*$
 together with the pull-back perversities on each CS set having $X^*$ as intrinsic CS set,
 see \remref{rem:invariance}.
 But, if we start from a perversity $\ov{p}$ on $X$, which is the natural situation, we need to construct a perversity
 on $X^*$.
 To that end, we introduce the notion of  K-perversity  
 which allows the construction of a perversity $\nu_{*}\ov{p}$ on $X^*$ from $\ov{p}$.
 The condition (C) of \defref{def:llperversite} is the conerstone for this construction.
 But we also need some other properties for getting sufficient conditions giving the following result,
 which covers also the well-known situation of GM-perversities.

\begin{theorem} 
Let $X$ be a CS set and $\ov{p}\colon \cS_{X}\to \Z$ be a K-perversity.
Then, the intrinsic aggregation $\nu \colon X \to X^*$ 
 induces an isomorphism
$H_{*}^{\ov{p}}(X;G) \cong  H_{*}^{\nu_{*}\ov{p}}(X^*;G)$.
\end{theorem}

With an extra condition on $\ov{p}$, the isomorphism can be extended to the tame $\ov{p}$-intersection homology.
Let us also notice the existence of examples justifying the choice of the properties defining a K-perversity,
see \remref{rem:yes}.

\smallskip
We now address the more general situation of two stratifications,
$(X,\cS)$ and $(X,\cT)$, on the same topological space, $X$.
We say that $(X,\cS)$ is a \emph{refinement} of $(X,\cT)$ if the strata of $\cT$ are union of strata of $\cS$. In other words,
we consider a pair of stratifications such that the identity map defines a stratified map,
$\kappa\colon (X,\cS)\to (X,\cT)$.
The previous situation of the intrinsic CS set appears as a particular case of a refinement.
If we start from a perversity $\ov{p}$ on $X$ and the pullback perversity $\kappa^*\ov{p}$,
\corref{refin} shows that $\kappa$ induces an isomorphism
$H_{*}^{\kappa^*\ov{p}}(X,\cS;G) \cong  H_{*}^{\ov{p}}(X,\cT;G)$.

As quoted by the referee, \thmref{thm:invariancegenerale} concerns properties of ``pushforward perversities'',
$\ov{p}$ and $\nu_{*}\ov{p}$,
in the case of the intrinsic aggregation, $\nu\colon X\to X^*$.
On the other hand, \corref{refin}
is a statement about ``pullback perversities'' for any refinement of a stratification, which thus includes
the case of $\nu$. Thus a natural question is the existence of a \thmref{thm:invariancegenerale} for 
any refinements.
Such result cannot be obtained in a short way and needs an independant work. It will be done
in a forthcoming work.

 \medskip
 The topological invariance of $\ov{p}$-intersection homology for general perversities has already been studied in
 two papers.
 
 -- In \cite{MR2209151}, G. Friedman considers pseudomanifolds (with pseudoboundary, possibly empty),
  with codimension-dependent perversities and tame intersection homology.
 At first, let us observe that the topological invariance cannot exist without restriction under this generality:
  for any perversity $\ov{p}$ with $\ov p > \ov t$,  
$\ov{p}$-intersection homology is equal to the relative intersection homology of $X$ and its singular part (see \propref{pro:PerDir}).
Friedman proves that this phenomenon is in fact the only obstruction. He shows that
the Deligne sheaf is independent of the choice of a stratification if the singular set is fixed.
In our framework, this result corresponds to \corref{cor:greg}.
 
 -- In \cite{MR3175250}, G. Valette considers PL-pseudomanifolds (with pseudoboundary, possibly empty),
 with perversities defined on the sets of strata but with the restriction $\ov{p}\leq \ov{t}$. Let
  $(X,\cS)$ be a \emph{refinement} of $(X,\cT)$ (see \defref{def:refinement}), 
  $\ov{p}$ a perversity on $\cS$ and $\ov{q}$ a perversity on $\cT$,
  such that
  $\ov{q}(T)\leq \ov{p}(S)\leq \ov{q}(T)+\codim_{T}S$,
  for any $S\in \cS$, $T\in\cT$, $S\subset T$. Then, Valette proves the existence of an isomorphism
  $$H_{*}^{\ov{p}}(X,\cS;G) \cong  H_{*}^{\ov{q}}(X,\cT;G).$$
The previous question concerning the existence of a generalization of  \thmref{thm:invariancegenerale} to 
any refinement corresponds to this setting, without the PL hypothesis and the restriction $\ov{p}\leq \ov{t}$.

\medskip
The homologies are with coefficients in an abelian group~$G$. In general, we do not explicit them in the proofs.

\tableofcontents

\setcounter{theorem}{0}
\section{Stratified spaces.}\label{sec:rappel}

\begin{quote}
In this section we recall the notions of  stratified spaces, CS sets and pseudomanifolds.
\end{quote}

\begin{definition}\label{def:espacefiltre} 
A \emph{filtered space} is a Hausdorff space,
$X$, endowed with a filtration by closed subspaces,
$$X_0\subseteq X_1\subseteq\ldots\subseteq X_n=X,$$
with  $X_n\backslash X_{n-1}\ne \emptyset$. The
\emph{dimension} of $X$ is denoted by $\dim X=n$. 
The connected components, $S$, of $X_{i}\backslash X_{i-1}$ are the \emph{strata} of  $X$ 
and we write $\dim S=i$ and $\codim S=\dim X-\dim S$. 
(These dimensions can be formal and not necessarily related to a notion of geometrical dimension.)

The strata of  $X_n\backslash X_{n-1}$ are \emph{regular strata}; the other ones are \emph{singular strata}.  
The family of non-empty strata is denoted by $\cS_{X}$ (or $\cS$ if there is no ambiguity). 
The subspace $\Sigma _{X}=X_{n-1}$ is  the \emph{singular set,} sometimes also denoted $\Sigma$.
\end{definition} 

\begin{exemple}\label{exem:espacefiltre}
Let $X$ be a filtered space of dimension $n$.
\begin{itemize}
\item An open subset $U \subset X$, endowed with the \emph{induced filtration} given by
$U_i = U \cap X_{i}$, becomes a filtered space.

\item If $M$ is a topological manifold, the \emph{product filtration} is defined by
$\left(M \times X\right) _i = M \times X_{i}$. The product  $M\times X$ becomes a filtered space.
\item If $X$ is compact, the open cone 
$\rc X = X \times [0,1[ \big/ X \times \{ 0 \}$ is endowed with the 
\emph{conical filtration} defined by
$\left(\rc X\right) _i =\rc X_{i-1}$,  $0\leq i\leq n+1$. By convention, 
$\rc \,\emptyset=\{ \tv \}$, where $\tv=[-,0]$ is the apex of the cone.
\end{itemize}
\end{exemple}

We enhance the notion of filtered space in order to obtain continuous maps with richer properties.

\begin{definition}\label{def:espacestratifie}
A \emph{stratified space} is a filtered space such that any pair of strata, $S$ and $S'$ with $S\cap \overline{S'}\neq \emptyset$,  verifies
$S\subset \ov{S'}$.
\end{definition}

The family $\cS$ is a poset  (\cite[Section A.2]{CST1}) or  (\cite[Section 2.2.1]{FriedmanBook})  relatively to  
$S \preceq S'$ if $S\subseteq \ov{S}'$. 
We write $S\prec S'$ if  $S\preceq S'$ and $S\neq S'$.

\begin{definition}\label{def:profondeur}
Let $X$ be a stratified space. 
 The \emph{depth of $X$}    is the greater integer $m$ for which it exists a chain of strata,
 $S_{0}\prec S_{1}\prec \cdots\prec S_{m} $.
 We denote  $\depth X=m$.
\end{definition}

\begin{exemple}\label{exem:espacestrtaifie} 
If  $X$ is stratified, each construction of \exemref{exem:espacefiltre} is a stratified space.
\end{exemple}

\begin{definition}\label{def:applistratifiee}
A \emph{stratified map,} $f\colon X\to Y$, is a continuous map between stratified spaces such that, for each stratum  $S\in\cS_{X}$,
there exists a unique stratum $\bi S f\in \cS_Y$ with $f(S)\subset \bi S f$ and
$\codim \bi S f\leq \codim S$.
\end{definition}

Observe that a continuous map $f\colon X\to Y$ is  stratified  if, and only if, the pull-back of a stratum 
  $S' \in \cS_Y$ is empty or a union 
$f^{-1}(S')=\sqcup_{i\in I}S_{i}$, with
$\codim S'\leq \codim S_{i}$ for each $i\in I$.
Therefore, a stratified map sends a regular stratum in a regular one but  the image of a singular stratum 
can be included in a regular one.

\begin{exemple}
Let $X$ be a stratified space.
The canonical projection,
$\pr\colon M\times X\to X$, 
 the maps
$\iota_{t}\colon X\to\rc X$ with $x\mapsto [x,t]$,
$\iota_{m}\colon X\to M\times X$ with $x\mapsto (m,x)$
and the canonical injection of an open subset  $U\hookrightarrow X$,
are stratified for the structures described in \exemref{exem:espacefiltre}.
\end{exemple}

Let us recall some properties of stratified maps from \cite[Section A.2]{CST1}.

\begin{proposition}[{\cite[Proposition A.23]{CST1}}]\label{prop:stratifieetordre}
A stratified map, $f\colon X\to Y$, induces the order preserving map
$(\cS_X,\preceq)\to (\cS_{Y},\preceq)$, defined by $S\mapsto \bi S f$.
\end{proposition}

Let us introduce the notion of homotopy between stratified maps. 
Here, the product $X\times [0,1]$ is endowed with the product filtration.

\begin{definition}\label{def:homotopie}
Two stratified maps 
$f,\,g\colon X\to Y$
are \emph{homotopic} if there exists a stratified map,  
$\varphi\colon X\times [0,1]\to Y$,
such that  $\varphi(-,0)=f$ and $\varphi(-,1)=g$.
Homotopy is an equivalence relation and produces the notion of homotopy equivalence between stratified spaces.
\end{definition}

The following notion of locally cone-like stratified space has been introduced by  Siebenman, \cite{Sieb}.

\begin{definition}\label{def:csset}
A \emph{CS set} of dimension $n$ is a filtered space,
$$
\emptyset\subset X_0 \subseteq X_1 \subseteq \cdots \subseteq X_{n-2} \subseteq X_{n-1} \subsetneqq X_n =X,
$$
such that, for each $i$, 
$X_i\backslash X_{i-1}$ is a topological manifold of dimension $i$ or the empty set. 
Moreover, for each point $x \in X_i \backslash X_{i-1}$, $i\neq n$, there exist
\begin{enumerate}[(i)]
\item an open neighborhood $V$ of  $x$ in $X$, endowed with the induced filtration,
\item an open neighborhood $U$ of $x$ in $X_i\backslash X_{i-1}$, 
\item a compact filtered space, $L$,
of dimension $n-i-1$, where the open cone, $\rc L$, is provided with the {\em conical filtration}, $(\rc L)_{j}=\rc L_{j-1}$.
\item a homeomorphism, $\varphi \colon U \times \rc L\to V$, 
such that
\begin{enumerate}[(a)]
\item $\varphi(u,\tv)=u$, for each $u\in U$, 
\item $\varphi(U\times \rc L_{j})=V\cap X_{i+j+1}$, for each $j\in \{0,\ldots,n-i-1\}$.
\end{enumerate}
\end{enumerate}
The pair  $(V,\varphi)$ is a   \emph{conical chart} of $x$
 and the filtered space $L$ is a \emph{link} of $x$. 
The CS set $X$ is called  \emph{normal} if its links are connected.
\end{definition}

In the above definition, the links are always non-empty sets.
Therefore, the open subset $ X_ {n} \backslash X_ {n-1} $ is dense.
Links are not necessarily CS sets but they are always filtered spaces.
Note also that the links associated to points living in the same stratum may not be  homeomorphic
but they always have the same intersection homology, see for example \cite [Corollary 5.3.14] {FriedmanBook}.
An open subset of the CS set $X$ and the product of $X$ with a topological manifold, 
endowed with the induced structures of \exemref{exem:espacefiltre},  are CS sets. This is also the case for the cone 
$\rc X$ when $X$ is compact.
Recall also that a CS set is a stratified space, see \cite[Theorem G]{CST1} for instance.

\begin{proposition}\label{prop:paracompact}
A paracompact CS set is metrizable.
\end{proposition}

\begin{proof}
Let  $X$ be a paracompact CS set. 
From a Smirnov theorem  \cite{MR0041420}, with the hypothesis ``paracompact'', it is sufficient 
to prove that $X$ is locally metrizable.
We proceed by induction on the depth of $X$. When $\depth(X)=0$, it is a consequence of the fact 
that $X$ is a topological manifold.

Let us suppose that the result is true for any paracompact CS set whose depth is  smaller than $\depth (X)$. 
Since the studied property is local, we can suppose that $X$ is a conical chart $U \times \rc L$,
with $U = \R^n$.
Let us notice that 
$U\times \rc L\backslash (U\times \{\tv\})=U\times L\times ]0,1[$
is an open subset of $X$ and therefore a  CS set of depth equal to  $\depth(X) -1$.
The product $U\times ]0,1[$ is metrizable, as subset of $\R^{n+1}$, thus it is paracompact
(\cite[Theorem 20.9]{Wil}).
Now, the product of a paracompact space with a Hausdorff compact space being paracompact
(\cite[Theorem 20.12.c]{Wil}), the product $U\times L\times ]0,1[$
is a paracompact CS set and therefore a metrizable space by induction hypothesis.

Hence the product $U\times L \times [1/3,2/3]$ is metrizable. 
The homeomorphic space $U\times L \times [0,1]$ is also metrizable.
So, the  product $ [-1,1]^n\times L\times [0,1]$ is a compact metrizable space. 
Since the image by a continuous map of a compact metrizable space in a Hausdorff space
 is a metrizable space (\cite[Corollary 23.2]{Wil}), 
 we get that $ [-1,1]^n\times (L\times [0,1]/L \times \{0\})$ is a metrizable space. 
 We conclude that  $]-1,1[^n\times \rc L$ is  metrizable.
Thus, the homeomorphic space $U\times \rc L $ is also metrizable.
\end{proof}

Pseudomanifolds are special cases of CS sets. Their definition varies in the literature;
we consider here the original definition of Goresky and MacPherson \cite{GM1}, 
without restriction on codimension one strata.

\begin{definition}\label{def:pseudomanifold}
An \emph{$n$-dimensional  pseudomanifold } (or simply pseudomanifold) is an $n$-dimensional CS set, 
for which  the link of a point $x \in X_i\menos X_{i-1}$ is an $(n-i-1)$-dimensional pseudomanifold.
\end{definition}

\section{Perversities}\label{subsec:perversite}

\begin{quote}
The main ingredient in intersection homology is the notion of perversity
which ``controls'' the transverse degree between simplices and strata.
Here we recall the various aspects of this notion.
\end{quote}

Let us begin with the original definition of perversity \cite{GM1}, 
called \emph{Goresky-Mac\-Pherson perversity,} to make a difference with the more general ones introduced in \cite{MacPherson90},
that we  simply call \emph{perversities}.

\begin{definition}\label{def:perversite} 
A \emph{loose perversity}  \cite{MR800845} is a map  $\ov{p}\colon \N\to\Z$, $i\mapsto \ov{p}(i)$, with $\ov{p}(0)=0$. 
A  \emph{Goresky-MacPherson perversity} is a loose perversity such that
\begin{equation}\label{equa:king}
\ov{p}(i)\leq\ov{p}(i+1)\leq\ov{p}(i)+1, \quad\text{for each} \quad i\geq 1,
\end{equation}
and $\ov{p}(1)=\ov{p}(2)=0$.
A  \emph{King perversity} is a loose perversity verifying only (\ref{equa:king}).
\end{definition}

In \cite{MR800845}, H. King observed that the condition (\ref{equa:king}) 
is necessary for the topological invariance
of intersection homology and 
that  (\ref{equa:king}) together  with $\ov{p}(1)\geq 0$ are sufficient. This
property being the main subject of this work, we particularize these perversities in the previous definition. 
Observe also that they have already been considered in the literature: the
super-perversities of \cite{MR2276609} are  King perversities verifying $\ov{p}(2)> 0$;
they cover the case $\ov{p}(2)=1$ of \cite{MR1127478}.

\medskip
Let us introduce  more general perversities. Unlike the previous ones,  they are defined on the set
of non-empty strata of the filtered space.

\begin{definition}\label{def:perversitegen}
A \emph{perversity on a filtered space} $X$, is a map,
$\ov{p}\colon \cS\to \Z$, taking the value~0 on the regular strata.
The pair $(X,\ov{p})$ is a  \emph{perverse space}.

When $X$ is a stratified space, or a CS set, we  say that $(X,\ov{p})$ is a  \emph{perverse stratified space} or
a \emph{perverse CS set.}
\end{definition}

Let  $X$ be a filtered space. A loose perversity in the sense of \defref{def:perversite} induces a perversity on $X$ by 
setting
$\ov{p}(S)=\ov{p}(\codim S)$.

If $\ov{p}$ and $\ov{q}$ are two perversities on $X$, we write  $\ov{p}\leq \ov{q}$ when
$\ov{p}(S)\leq\ov{q}(S)$, for each $S\in \cS$. 
The poset of Goresky-MacPherson perversities possesses a maximal element: the \emph{top perversity} $\ov t$, defined by 
 $\ov{t}(i)=i-2$, if $i\geq 2$.  By extension, we shall denote by $\ov t$ the perversity on $X$ 
 defined by $\ov{t}(S)=\codim S-2$, if $S$ is a singular stratum. 
 Given a perversity $\ov p$ on $X$,  the \emph{complementary perversity} on $X$, $D\ov{p}$, is characterized 
 by $D\ov{p}+\ov{p}=\ov{t}$.
 Mention also the \emph{null perversity}, $\ov{0}$, defined by $\ov{0}(S)=0$.

\begin{definition}\label{def:perversiteimagereciproque}
Let $f\colon X\to Y$ be a stratified map and $\ov{q}$ a perversity on $Y$.
 The  \emph{pull-back perversity of $\ov{q}$ by $f$} is the perversity $f^*\ov q$ on $X$ defined by
$(f^*\ov{q})(S)=\ov{q}(\bi S f)$,
for each $S\in \cS_{X}$.
\end{definition}

\begin{exemple}\label{exem:perversinduits}
Let $(X,\ov{p})$ be a perverse stratified space and consider the constructions of \exemref{exem:espacefiltre}.
\begin{itemize}
\item Any  open subset $U \subset X$ is endowed with the pull-back perversity of $\ov{p}$ by the canonical inclusion $U \hookrightarrow X$.
\item The  product with a topological manifold, $X\times M$, is endowed with
the pull-back perversity of $\ov{p}$ by the canonical projection $X\times M\to X$. We also denote it $\ov{p}$.
\item When $X$ is compact, we have $\cS_{\rc X}=\left\{ S\times ]0,1[\mid S\in \cS_{X}\right\}\cup\left\{\{\tv\}\right\}$.
So, a perversity $\ov{q}$ on   $\rc X$ induces a perversity on $X$ defined by
$\ov{q}(S)=\ov{q}(S\times ]0,1[)$. We also denote it  $\ov{q}$.
\end{itemize}
\end{exemple}

\begin{remarque}\label{rem:pourquoigenerales}
Let us consider the following classical result: free actions of the circle  on a topological space,
 $X$, are classified by the orbit space, $B=X/ S^1$, and the Euler class $e\in H^2(B;\Z)$.
 An extension of this result to the case of not necessarily free actions has been developed in 
 \cite{MR2344739, MR3183107,MR2164356} in the case of a modeled action (see \cite{MR2344739})
  and of pseudomanifolds,  $X$ and $B$. 
In  \cite{MR2344739}, a perversity $\ov{e}$ and an Euler class $e$,
belonging to the intersection cohomology group $H_{\ov{e}}^2(B;\R)$, are defined. As in the case of a free action, 
the pseudomanifold  $B$ and the Euler class $e$  determine the intersection cohomology of $X$.
The Euler perversity, $\ov{e}$, is not a loose perversity. It takes the values  0, 1 or 2,
following the behavior of the circle on the stratum:
\begin{itemize}
\item $\ov{e}(S)= 0$ if the stratum $S$ has not fixed points,
\item $\ov{e}(S)= 1$ if the stratum  $S$  has only fixed points  and  the circle   acts trivially on the link $L_{S}$ of $S$,
\item $\ov{e}(S)= 2$ on the other strata.
\end{itemize}
This example shows the interest for  more general perversities than the classical ones.
\end{remarque}

\section{Intersection homologies.  Theorems \ref{thm:martinmacphersonhomologie} and \ref{thm:martingreghomologie}}
\label{subsec:homologie}

\begin{quote}
We present two chain complexes computing the intersection homology of Goresky and MacPherson \cite{GM1}.  The first one  comes from the work of King 
\cite{ MR800845}. The second one was already present  in several works of the second author, cf.~\cite{MR2210257} for example.
We also give a tame version of these homologies.
\end{quote}

Let us begin with the notion of allowable simplex.

\begin{definition}\label{def:allowableMacPherson}
Let   $(X, \ov{p})$ be a perverse space. A singular simplex  $\sigma\colon \Delta\to X$, is \emph{${\ov{p}}$-allowable} if, for each stratum $S\in \cS$, the subset $\sigma^{-1}(S)$ is included in the skeleton of $\Delta$ of dimension,
\begin{equation}\label{equa:allowable}
\dim \Delta-\codim S+\ov{p}(S).
\end{equation}
A singular chain $\xi$ is a \emph{$\ov p$-allowable chain} if any simplex with a non-zero coefficient in $\xi$ is 
$\ov p$-allowable. It is a \emph{$\ov p$-intersection chain} if $\xi$ and $\partial \xi$ are $\ov p$-allowable chains.
Denote by  $I^{\ov{p}}C_{*}(X;G)$ the chain complex of $\ov p$-intersection chains and by 
$I^{\ov{p}}H_{*}(X;G)$ its homology.
\end{definition}

Let us observe that the condition \eqref{equa:allowable} is satisfied for any regular stratum since
$\codim S=\ov{p}(S)=0$.
In \cite{MR800845}, King
 proved that  $I^{\ov{p}}H_{*}(X;G)$ is isomorphic to the original intersection homology of Goresky and MacPherson introduced in \cite{GM1}.

\medskip
Let us introduce the filtered simplices, more adapted to a simplicial approach of intersection homology.

\begin{definition}\label{def:filteredsimplex}
Let $X$ be a filtered space of dimension $n$.  
A  \emph{filtered simplex} is a singular simplex  $\sigma\colon\Delta\to X$ endowed with a filtration
 $\Delta=\Delta_{0}\ast\Delta_{1}\ast\cdots\ast\Delta_{n}$, called  \emph{$\sigma$-decomposition of $\Delta$},
verifying$$
\sigma^{-1}X_{i} = \Delta_{0}\ast\Delta_{1}\ast\cdots\ast\Delta_{i},
$$
for each~$i \in \{0, \ldots, n\}$.  Here, $*$ denotes the join of two simplices, with the convention $\emptyset * Y=Y$,
and we require that the join decomposition,
 $\Delta=\Delta_{0}\ast\Delta_{1}\ast\cdots\ast\Delta_{n}$, 
 is compatible with the order of vertices.
\end{definition}

In other words, a singular simplex $\sigma \colon \Delta \to X$ is a filtered simplex if  $\sigma^{-1}(X_i)$, 
 $i\in \{0,\ldots,n\}$, is a face of $\Delta$ or the emptyset.
 The dimension of the simplices $\Delta_{i}$ measures the defect of transversality 
 of $\sigma$ relatively to the strata of $X$.

\begin{definition}\label{def:perversedegree}
Let $X$ be a filtered space of dimension $n$ and  
$\sigma\colon\Delta=\Delta_{0}\ast \cdots\ast\Delta_{n}\to X$ be a filtered simplex.
\begin{enumerate}[{\rm (i)}]
\item The \emph{perverse degree of} $\sigma$ is the  $(n+1)$-tuple,
$\|\sigma\|=(\|\sigma\|_0,\ldots,\|\sigma\|_n)$,  
where
 $\|\sigma\|_{i}=\dim \sigma^{-1}X_{n-i}=\dim (\Delta_{0}\ast\cdots\ast\Delta_{n-i})$, 
 with the convention $\dim \emptyset=-\infty$.

 \item For each stratum  $S \in \cS$, the \emph{perverse degree of  $\sigma$ along the stratum  $S$} is 
 $$\|\sigma\|_{S}=\left\{
 \begin{array}{cl}
 -\infty,&\text{if } S\cap \sigma(\Delta)=\emptyset,\\
 \|\sigma\|_{\codim S},&\text{if not.}
  \end{array}\right.$$

  \item For each stratum  $S \in \cS$, the  \emph{perverse degree of the chain  
  $\xi=\sum_{j\in J}\lambda_{j}\sigma_{j}$, $\lambda_{j}\neq 0$,  along the stratum $S$} is
  $\|\xi\|_{S}=\max_{j\in J}\|\sigma_{j}\|_{S}.$
\end{enumerate}
  The map $\|\xi\|\colon \cS\to \N\cup\{-\infty\}$ is defined by $S  \mapsto \|\xi\|_{S}$. 
\end{definition}

\begin{remarque}\label{rem:allowablefiltre}
Let  $\sigma\colon \Delta\to X$ be a filtered simplex and   $S$ be a stratum of $X$.
If $\sigma^{-1}(S) \ne \emptyset$, we denote $F_{S}$ the smallest face of $\Delta$ 
containing $\sigma^{-1}(S)$.
Following \defref{def:allowableMacPherson}, the filtered simplex $\sigma$ is $\ov{p}$-allowable if, and only if, 
$\sigma^{-1}(S) = \emptyset$ or
$  \dim F_{S}\leq \dim \Delta - \codim S + \ov{p}(S)$, for each $S \in \cS$.

Denote $i=\codim S$. We  proved in  \cite[Lemme A.24]{CST1} that
  $
\sigma^{-1}(S)=\emptyset$ or
$
\sigma^{-1}(S) = \Delta_{0}* \cdots * \Delta_{n-i} \backslash \Delta_{0}* \cdots * \Delta_{n-i -1}
$.
In the second case we have $\dim F_{S}=\|\sigma\|_{i}=\|\sigma\|_{S}$ and therefore we get:
\begin{equation*}\label{equa:filtreallowable}
\sigma \text{ is } \ov{p}{\text{-allowable}} \Longleftrightarrow
\|\sigma\|_{S}\leq \dim \Delta - \codim S + \ov{p}(S), \ \ \forall S\in\cS.
\end{equation*}
\end{remarque}

 \begin{definition}\label{def:chaineintersection} 
Let  $(X, \ov{p})$ be a perverse space. We denote $C_{*}^{\ov{p}}(X;G)$ 
the chain complex of  $\ov{p}$-intersection chains made up of filtered simplices and
$H^{\ov{p}}_*(X;G)$ its homology.
\end{definition}

The next theorem establishes the existence of an isomorphism between the two complexes introduced above,
in \defref{def:allowableMacPherson} and \defref{def:chaineintersection}.

\begin{theorem}\label{thm:martinmacphersonhomologie}
Let $(X,\ov{p})$ be a perverse CS set.
The canonical inclusion 
$C_{*}^{\ov{p}}(X;G)\hookrightarrow I^{\ov{p}}C_{*}(X;G)$
is a chain map inducing an isomorphism in  homology, 
$$H^{\ov{p}}_*(X;G)\cong I^{\ov{p}}H_{*}(X;G).$$
\end{theorem}

For Goresky-MacPherson perversities this result is \cite[Proposition A.29]{CST1}.
We find in \cite[Proposition 2.4.1]{MR2210257} a slightly different version using perversities.
There, a blow-up of $X$ is used and this process needs extra conditions on the space $X$. (For example, 
the blow-up of $X$ exists if  $X$ is a Thom-Mather space.)
We give in  \secref{sec:thmA} a proof of  \thmref{thm:martinmacphersonhomologie}, 
without this restriction.

\bigskip

We introduce now the \emph{tame intersection homology.} 
When we are dealing with perversities such that $\ov{p}\not\leq\ov{t}$, one needs 
this variation for having
a deRham Theorem \cite{MR2210257} or  a Poincar\'e Duality \cite{CST2,FM}.
More precisely, a perversity $\ov{p}$ with $\ov{p}\leq\ov{t}$ has an essential property coming from
(\ref{equa:allowable}):
 any $\ov p$-allowable simplex $\sigma \colon \Delta\to X$ verifies  
$\sigma(\partial\Delta)\not\subset \Sigma$. But if $\ov{p}\not\leq\ov{t}$ this kind of simplices may exist 
with the consequence that deRham Theorem and  Poincar\'e Duality fail. Tame intersection homology is defined 
from
the elimination of the allowable simplices included in the singular part.
Let us begin with the notion of tame simplex (called non-GM in \cite{FriedmanBook}).

\begin{definition}\label{def:tameGreg}
 Let $X$ be a filtered space. A singular simplex $\sigma  \colon \Delta  \to X$ is \emph{regular} 
 if  $\sigma(\Delta) \not\subset \Sigma$. 
 In the boundary  $\partial \sigma=\sum_{i=0}^m (-1)^i\sigma_{i}$, we keep only the regular simplices $\{\sigma_{i}\}_{i\in \gI}$ and  set
 $$\gd \sigma=\sum_{i\in\gI}(-1)^i\sigma_{i}.$$ 
Notice that  $\gd^2=0$.
The chain complex  
$I \gC_{*}(X;G)$ 
is the complex generated by the regular simplices of $X$ endowed with the differential 
$\gd$.
\end{definition}

And now come the perversities.

 \begin{definition}\label{def:homologiedroite}
  Let $(X,\ov p)$ be a perverse space. 
  A \emph{$\ov p$-tame simplex} is a $\ov p$-allowable regular simplex.
A chain  
$\xi \in I \gC_{*}(X;G)$ 
is a 
  \emph{$\ov p$-tame chain} if we can write it as a combination of $\ov p$-tame simplices. 
  A $\ov p$-tame chain  $\xi$ such that $\gd \xi$ is $\ov p$-tame is called a \emph{$\ov p$-tame intersection chain}. The complex of $\ov p$-tame intersection chains is denoted by 
  $I ^{\ov{p}}\gC_{*}(X;G)$. Its homology,
   $I ^{\ov{p}}\gH_{*}(X;G)$,
  is  the  \emph{$\ov{p}$-tame intersection  homology} of $X$.
 \end{definition}

We adapt this  definition to filtered simplices as follows.

 \begin{definition}\label{def:chaineintersectiontame} 
Let  $(X, \ov{p})$ be a perverse space. We denote 
$ \gC_*^{\ov p} (X;G )$ 
the chain complex of  $\ov{p}$-tame intersection chains made up of filtered simplices. Its homology is 
 $ \gH_*^{\ov p} (X;G )$ 
\end{definition}

\begin{remarque}\label{igual}
Condition \eqref{equa:allowable} implies that the $\partial$-boundary of a $\ov p$-allowable simplex does not live in $\Sigma$, when $\ov p \leq \ov t$. So, in this case, we have 
$I^{\ov{p}} \gC_* (X;G )=I^{\ov{p}} C_*(X;G )$,
$ \gC _*^{\ov p} (X;G )= C_*^{\ov p} (X;G )$,
$I^{\ov{p}} \gH_* (X;G )=I^{\ov{p}} H_*(X;G )$
and
$ \gH_*^{\ov p} (X;G )= H_*^{\ov p} (X;G )$
\end{remarque}

This homology is isomorphic to the non-GM intersection homology of Friedman~\cite{FriedmanBook}.

\begin{theorem}\label{thm:martingreghomologie}
Let $(X,\ov{p})$ be a perverse CS set.
The canonical inclusion 
$ \gC _*^{\ov p} (X;G )\hookrightarrow I^{\ov{p}} \gC_* (X;G )$
is a chain map inducing an isomorphism in  homology,
$$ \gH_*^{\ov p} (X;G )\cong I^{\ov{p}} \gH_* (X;G ).$$
\end{theorem}

The proof of  \thmref{thm:martingreghomologie} is postpone to \secref{sec:thmA}. 
We end this Section with some properties of stratified maps relatively to 
$H_*^{\ov p} (X;G )$ and $\gH_*^{\ov p} (X;G )$.
We need the following notion introduced in \cite{MR2210257}.

\begin{definition}\label{def:stratifieesubtil}
Let $(X,\ov p)$ be a perverse CS set. We set:
$$X_{\ov{p}}= \cup \{ \ov{S} \mid S \in \cS \text{ singular stratum with $\ov p (S) > \ov t (S)$} \}.$$ 
\end{definition}

 \begin{proposition}\label{prop:stratifieethomotopie}
Let $f\colon (X,\ov{p})\to (Y,\ov{q})$ be  a stratified map between two perverse stratified spaces with  $f^*D\ov q \leq D\ov p$.
The association $\sigma\mapsto f\circ \sigma$ 
defines a chain map
$f_{*}\colon C_*^{\ov p} (X;G )\to C_*^{\ov q} (Y;G )$.
Moreover, if $f (X_{\ov p}) \subset \Sigma_Y$, then it induces also a chain map
$f_{*}\colon \gC_*^{\ov p} (X;G )\to \gC_*^{\ov q} (Y;G )$.
\end{proposition}

Since the map $f$ is stratified, we observe that the condition  $\ov p \leq f^*\ov q$ implies 
$f^*D\ov q \leq D\ov p$. We also remark that the condition $f(\Sigma_{X})\subset \Sigma_{Y}$
implies $f (X_{\ov p}) \subset \Sigma_Y$.

\begin{proof}
Let 
$\sigma\colon \Delta=\Delta_{0}\ast\cdots\ast\Delta_{n}\to X$
be a $\ov{p}$-allowable filtered simplex.
We proceed in two steps.

$\bullet$ \emph{The map $f_*$ preserves allowability;} i.e.,
\begin{equation*}\label{equa:imageallowable}
\sigma \  \ov{p}\text{-allowable filtered  simplex} \Rightarrow  f_*(\sigma) \  \ov{q}\text{-allowable filtered simplex}.
\end{equation*}
 By definition, the simplex $\sigma$ verifies 
\begin{equation}\label{equadansX}
\|\sigma\|_{S}\leq \dim \Delta -\codim S+\ov{p}(S),
\end{equation}
for each stratum $S \in \cS_{X}$ and we have to prove 
\begin{equation}\label{equadansY}
\|f\circ\sigma\|_{T}\leq \dim\Delta-\codim T+\ov{q}(T)
\end{equation}
for each singular stratum  $T$ of $Y$.

If $(f\circ \sigma)(\Delta)\cap T=\emptyset$ then $\|f\circ\sigma\|_{T}=-\infty$ and 
(\ref{equadansY}) is verified. We can suppose  $(f\circ \sigma)(\Delta)\cap T\neq\emptyset$.
We know, from
\cite[Lemma A.24]{CST1} and the
beginning of the proof of \cite[Theorem F]{CST1},
 that the strata $T$ of $Y$ verifying
$(f\circ \sigma)(\Delta)\cap T\neq\emptyset$
are totally ordered,
$T_{1} \prec T_{2} \prec \dots \prec T_{p}$
and that there exists a family of strata of $X$,
$$S_{1} \prec \dots \prec S_{k_{1}} \prec S_{k_{1}+1} \prec\dots \prec S_{k_{2}}\prec\dots\prec S_{k_{p}},$$
such that
$S^f_{k_{i}+j}=T_{i+1}$ for any $j\in \{1,\dots,k_{i+1}-k_{i}\}$.
We get
\begin{eqnarray*}
\|f\circ \sigma\|_{\codim T_{i}}
&=_{(1)}&
\dim (f\circ\sigma)^{-1} (Y_{m-\codim T_{i}})
=\dim (f\circ\sigma)^{-1} (Y_{\dim T_{i}})
\\
&=_{(2)}&
\dim (f\circ\sigma)^{-1}(T_{1}\sqcup\dots\sqcup T_{i})
=
\dim \sigma^{-1}(S_{1}\sqcup\dots\sqcup S_{k_{i}})
\\
&=&
\dim (\Delta_{0}\ast\dots\ast\Delta_{\dim S_{k_{i}}})\\
&=_{(3)}&
\|\sigma\|_{\codim S_{k_{i}}},
\end{eqnarray*}
where $=_{(1)}$ and $=_{(3)}$ are \defref{def:perversedegree},
$=_{(2)}$ is the formula (36) in the statement of \cite[Lemma A.24]{CST1}.
As $\sigma$ is $\ov{p}$-allowable, we deduce
$$\|f\circ \sigma\|_{\codim T_{i}}\leq \dim \Delta-\codim S_{k_{i}} +\ov{p}(S_{k_{i}}).$$
On the other hand, by hypothesis, we have
$f^*D\ov{q}(S_{k_{i}})\leq D\ov{p}(S_{k_{i}})$
which means, by definition of $f^*D\ov{q}$,
$D\ov{q}(T_{i})\leq D\ov{p}(S_{k_{i}})$.
From the expression of the complementary perversity, this last inequality can be written as
$$\codim T_{i}-2-\ov{q}(T_{i})
\leq
\codim S_{k_{i}}-2-\ov{p}(S_{k_{i}}),
$$
which gives
$$\dim \Delta-\codim S_{k_{i}}+\ov{p}(S_{k_{i}})
\leq
\dim \Delta-\codim T_{i}+\ov{q}(T_{i}).$$
Finally, the inequality (\ref{equadansY}) is proven for $T_{i}$.

$\bullet$ \emph{The operator $f_*$ is differential}.
We have $f_* (\partial\sigma) = \partial f_*(\sigma)$ and it suffices to prove  
$f_* (\gd\sigma) = \gd f_*(\sigma)$ in the case of a regular simplex 
$\sigma \colon \Delta=\Delta_{0}\ast\dots\ast\Delta_{n} \to X$.
Let us notice, 
 \begin{equation}\label{equa:dgd}
 f_{*}(\partial \sigma -\gd \sigma)= f_{*}(\partial\sigma)-f_{*}(\gd\sigma)= \partial f_{*}(\sigma)-f_{*}(\gd\sigma).
 \end{equation}
 If we prove 
 $f_*(\partial \sigma- \gd \sigma) \subset \Sigma_Y$,
 then the regular part of (\ref{equa:dgd}) is equal to zero, which means exactly
 $\gd f_*(\sigma) - f_{*}(\gd \sigma)=0$.
Let us see the inclusion $f_*(\partial \sigma- \gd \sigma) \subset \Sigma_Y$.

Since $f (X_{\ov{p}}) \subset \Sigma_Y$, we are reduced to establish
$\im \sigma' \subset X_{\ov{p}}$ for each codimension one non-regular face $\sigma' \colon \Delta' \to X$ of $\sigma$.
By dimension reasons, we have $\dim \Delta_n = 0$ and also there exists   $i<n$ with  $\Delta_{i}\neq\emptyset$.
We denote $j$ the largest integer $i$ with this property.
Let $S \in \cS_{X}$  be the singular stratum  with $n-j = \codim S$ and $S\cap \sigma(\Delta)\neq\emptyset$.
The $\ov p$-allowability of $\sigma$ gives 
$\|\sigma\|_{n-j}=\dim(\Delta_{0}\ast\cdots\ast\Delta_{j})=\dim \Delta -1\leq \dim \Delta - \codim S+\ov{p}(S)
$.
Hence $\ov t (S)  = \codim S -2  <  \ov p(S)$ and therefore  $S \subset \ib X {\ov p}$, which gives the result.
\end{proof}

\begin{remarque}\label{rem:singisenough}
In the previous proof, the strata $T_{i}$ and therefore $S_{j}$ are singular. So, it suffices to verify the inequality
$f^*D\ov q (S) \leq D\ov p(S) $ for singular strata.
\end{remarque}

\begin{proposition}\label{prop:homotopesethomologie}
Let $\varphi\colon (X\times [0,1],\ov{p})\to (Y,\ov{q})$ be 
a homotopy between two stratified maps 
$f,\,g\colon (X,\ov{p})\to (Y,\ov{q})$ with 
$\varphi^*D\ov q \leq D\ov p$. Then 
 $f$ and $g$
induce the same map in homology,
$f_{*}=g_{*}\colon H_*^{\ov p} (X;G )\to H_*^{\ov q} (Y;G )$.
Moreover, if $\varphi (X_{\ov p} \times [0,1]) \subset \Sigma_Y$ then we also have
$f_{*}=g_{*}\colon \gH_*^{\ov p} (X;G )\to \gH_*^{\ov q} (Y;G )$.
\end{proposition}

\begin{proof} 
Consider the canonical injections,
 $\iota_{0},\,\iota_{1}\colon X\to X\times [0,1]$,
  defined by $\iota_{k}(x)=(x,k)$ for $k=0,\,1$. 
They are stratified maps and verify 
$\iota_k (X_{\ov p}) \subset \Sigma_{X\times [0,1]}$ for $k=0,\,1$.
 From
\propref{prop:stratifieethomotopie}, we have the homomorphism
$\varphi_{*}\colon H_*^{\ov p} (X\times [0,1] )\to H_*^{\ov q} (Y )$
and also
$\varphi_{*}\colon \gH_*^{\ov p} (X\times [0,1] )\to \gH_*^{\ov q} (Y )$
when
$\varphi (X_{\ov p} \times [0,1]) \subset \Sigma_Y$
Since 
$f=\varphi\circ\iota_{0}$ and $g=\varphi\circ\iota_{1}$ then it suffices to prove 
$\iota_{0,*}=\iota_{1,*}\colon  H_*^{\ov p} (X )\to H_*^{\ov p} (X\times [0,1] )$
and
$\iota_{0,*}=\iota_{1,*}\colon  \gH_*^{\ov p} (X )\to \gH_*^{\ov p} (X\times [0,1] )$.

Let $\sigma\colon \Delta=\langle e_{0},\ldots,e_{m}\rangle\to X$ be a
 filtered simplex.
The vertices of $\Delta \times [0,1]$ are $a_{j}=(e_{j},0)$ and $b_{j}=(e_{j},1)$. We define an $(m +1)$-chain on   $\Delta \times [0,1]$ by
$P=\displaystyle \sum_{j=0}^m (-1)^j \langle a_{0},\ldots,a_{j},b_{j},\ldots,b_{m}\rangle.$
This gives a chain homotopy 
$h\colon
C_{*}(X)\to C_{*+1}(X\times [0,1])$,
between $\iota_{0,*}$ and $\iota_{1,*}$, defined by
$\sigma\mapsto (\sigma\times\id)_{*}(P)$, where $C_{*}(-)$ denotes the complex generated by filtered simplices.
Let us also notice
$$
((\sigma\times\id)_{*}(P_{j}))^{-1}(X_{i}\times [0,1])=(\sigma^{-1}(X_{i})\times [0,1])\cap P_{j},
$$
with $P_{j}=
\langle a_{0},\ldots,a_{j},b_{j},\ldots,b_{m}\rangle
$.
Thus the simplices of $h$ are filtered, since $\sigma$ is.

Consider   the complex $ \gC_{*}(-)$, generated by regular simplices. Since
 $h$ preserves regular simplices and non-regular simplices, 
we  get the chain homotopy 
$h\colon
\gC_{*}(X)\to \gC_{*+1}(X\times [0,1])$
between $\iota_{0,*}$ and $\iota_{1,*}$,
Now, it remains to prove
\begin{equation}\label{7}
\sigma \  \ov{p}\text{-allowable filtered  simplex} \Rightarrow  h(\sigma) \  \ov{p}\text{-allowable filtered chain}.
\end{equation}
Let $j\in\{0,\ldots,m\}$. We denote
$\tau_{j}\colon\nabla=\langle v_{0},\ldots,v_{m+1}\rangle\to \Delta^m\times [0,1]$,
the simplex defined by
$(v_{0},\ldots,v_{m+1})\mapsto (a_{0},\ldots,a_{j},b_{j},\ldots,b_{m})$. By construction we have
$\Delta^m\times [0,1]=\cup_{j=0}^m\tau_{j}(\nabla)$.
For each stratum $S \in \cS$, we have
$$\tau_{j}^{-1}(\sigma\times\id)^{-1}(S\times [0,1])=
\tau_{j}^{-1}(\sigma^{-1}(S)\times [0,1])\subset
\sigma^{-1}(S)\times [0,1],$$
and therefore
$$
\|(\sigma\times\id)\circ\tau_{j}\|_{S\times [0,1]}\leq 
\|\sigma\|_{S}+1
\leq 
m- \codim S+\ov{p}(S)+1.
$$
We get
  \begin{eqnarray*}
  \|h(\sigma)\|_{S \times [0,1]}
  &\leq &
  \max_{j}\|(\sigma\times \id)\circ \tau_{j}\|_{S \times [0,1]}
  \leq  m- \codim S+\ov{p}(S)+1 \\
  &= &
 \dim\nabla-\codim (S\times [0,1]) +\ov{p}(S \times [0,1]).
   \end{eqnarray*}
This proves \eqref{7}.
 \end{proof}

\begin{corollaire}\label{cor:RfoisX}
Let $(X,\ov{p})$ be a perverse space. 
The inclusions $\iota_{z}\colon X\hookrightarrow \R\times X$, $x\mapsto (z,x)$,  with $z\in\R$, 
and the canonical projection  $\pr \colon \R\times X\to X$ induce isomorphisms
$H^{\ov{p}}_{k}(X;G)\cong H^{\ov{p}}_{k}(\R\times X;G)$
and
$\gH^{\ov{p}}_{k}(X;G)\cong \gH^{\ov{p}}_{k}(\R\times X;G)$.
\end{corollaire}
\section{Mayer-Vietoris sequence}\label{sec:MVhomologie}

\begin{quote}
In this section, we construct  Mayer-Vietoris sequences for the intersection homology and the tame intersection homology, 
$H^{\ov{p}}_{*}(X;G)$ and $\gH^{\ov{p}}_{*}(X;G)$. 
This is  a key point in  the proofs of 
Theorems \ref{thm:martinmacphersonhomologie}, \ref{thm:martingreghomologie}
and \ref{thm:invariancegenerale}.
\end{quote}

\begin{proposition} [Mayer-Vietoris sequence]\label{prop:MVhomologie}
Given a perverse space   $(X, \ov{p})$ and an open covering  $\{U,V\}$ of  $X$,
there exist two exact sequences,
\begin{equation}\label{equa:MVH}
\def\objectstyle{\scriptstyle}
\xymatrix@C=5mm{	
\dots \ar[r] & H^{\ov{p}}_{i}(U\cap V;G)
\ar[r] & 
H^{\ov{p}}_{i}(U;G)\oplus H^{\ov{p}}_{i}(V;G)
\ar[r] & 
 H^{\ov{p}}_{i}(X;G)
 \ar[r] & 
  H^{\ov{p}}_{i-1}(U\cap V;G)
 \ar[r] & \dots\\
}\end{equation}
and
\begin{equation}\label{equa:MVgH}
\def\objectstyle{\scriptstyle}
\xymatrix@C=5mm{	
\dots \ar[r] & \gH^{\ov{p}}_{i}(U\cap V;G)
\ar[r] & 
\gH^{\ov{p}}_{i}(U;G)\oplus \gH^{\ov{p}}_{i}(V;G)
\ar[r] & 
 \gH^{\ov{p}}_{i}(X;G)
 \ar[r] & 
  \gH^{\ov{p}}_{i-1}(U\cap V;G)
 \ar[r] & \dots \\
}\end{equation}
The connecting map of (\ref{equa:MVH}) is given by 
$\delta_{h}(\xi)=[\partial \xi_{U}]$,
where $\xi_{U}$ is obtained from the subdivision operator,
$\sd^k\xi=\xi_{U}+\xi_{V}\in \lau C {\ov{p}}* U + \lau C {\ov{p}} *V$ of \eqref{equa:subdivisionetintersection}.
The connecting map of (\ref{equa:MVgH}) is similar from the subdivision operator (\ref{equa:subdivisionetintersectionBis}).
\end{proposition}

Before proving this result we need to study the  $\ov{p}$-allowable simplices that are not 
$\ov{p}$-intersection chains.
We follow the method used in \cite[Proposition A.14.(i)]{CST1}. The thrust is that the $\ov p$-allowability default of the boundary of a $\ov p$-allowable simplex is  concentrated in only one  face.

\begin{definition}{(\cite[Definition A.11]{CST1})}\label{def:banefulfatal}
Let $(X,\ov p)$ be a perverse space and
  $\sigma \colon \Delta \to X$ be a filtered  $\ov{p}$-allowable simplex.  
For each stratum  $S \in \cS$ with $\sigma^{-1}(S)\neq\emptyset$, we denote
$F_S$ the smallest face of  $\Delta$ containing $\sigma^{-1}(S) $.
\emph{The $\ov{p}$-bad face} of $\sigma$ 
is the restriction 
$\sigma \colon T_{\ov{p}}(\sigma)\to X$ 
where
 \begin{equation}\label{52}
T_{\ov{p}}(\sigma) = \min \{ F_S \mid S\in\cS,  F_{S}\neq \Delta  \text{ \and }\dim F_S = \dim \Delta - \codim S + \ov{p}(S)\},
 \end{equation}
if this family is not empty. 
 By extension,  the face $T_{\ov{p}}(\sigma)$ of  $\Delta$ is also called a \emph{$\ov{p}$-bad face}.
 \end{definition}

Since the previous family  $\{F_S \}$  mentioned above is totally ordered (cf. \cite[Lemma A.24]{CST1}),
 the definition of  $T_{\ov{p}}(\sigma)$ makes sense. 
 Observe also that if the stratum  $S$ is regular then  $\ov{p}(S)=0=\codim S$ and 
 the inequality of \eqref{52} does not occur.

\begin{proposition}\label{prop:badface}
Let $(X,\ov p)$ be a perverse space and
$\sigma \colon \Delta \to X$  a filtered  $\ov{p}$-allowable simplex. 
\begin{enumerate}[\rm (a)]
\item  A codimension one face $s$ of  $\sigma$ is not  $\ov{p}$-allowable if, and only if,
$\sigma$ has a $\ov{p}$-bad face included in $s$.
\item Let  $\sigma'\colon \Delta\to X$ be another  $\ov{p}$-allowable simplex.
If $\sigma$ and  $\sigma'$ share  a codimension one  face  $\sigma''$ which is  not $\ov p$-allowable, then 
$\sigma$ and $\sigma'$ have the same  $\ov{p}$-bad face. Moreover, this face is a face of $\sigma''$.
\item The simplex  $\sigma$ is a $\ov{p}$-intersection chain if, and only if, it has no $\ov{p}$-bad face.
\end{enumerate}
\end{proposition}

\begin{proof}
(a) Let $s\colon \nabla \to X$ be a face of  codimension one of $\sigma$. 

$\bullet$ Let us suppose that the face  $s$ is not  $\ov{p}$-allowable. 
So, there exists a stratum  $S \in \cS$ such that  $s^{-1}(S)$ is not included in
the $( \dim \nabla - \codim S + \ov{p}(S))$- skeleton of  $\nabla$. 
Since  $s^{-1}(S) \ne \emptyset$ then $\sigma^{-1}(S) \ne \emptyset$. 
Recall that $F_{S}$ is the smallest face of  $\Delta$ containing $\sigma^{-1}(S)$. 
As the simplex $\sigma$ is $\ov{p}$-allowable, we have
$$\dim(F_{S}\cap \nabla)\leq\dim F_{S}\leq \dim \Delta - \codim S + \ov{p}(S).$$
On the other hand, as $s^{-1}(S)\subset F_{S}\cap \nabla$  and the simplex $s$ is not $\ov{p}$-allowable,
 we get
$$
\dim(F_{S}\cap \nabla)\geq \dim\nabla+1-\codim S+\ov{p}(S)
=  \dim \Delta - \codim S + \ov{p}(S).
$$ 
We conclude that  $\dim (F_{S}\cap\nabla)=\dim F_S = \dim \Delta - \codim S + \ov{p}(S)$ and
$F_S \subset \nabla $. 
As  $T_{\ov{p}}(\sigma)$ is the minimum for the order $\preceq$ of the faces $F_{S}$, we deduce  
$T_{\ov{p}}(\sigma) \subset \nabla$.

$\bullet$ Reciprocally, let us suppose that the $\ov{p}$-bad face of  $\sigma$ exists and verifies 
$T_{\ov{p}}(\sigma)  \subset \nabla$. So, there exists a stratum $S \in \cS$ with 
$$F_S \subset \nabla \text{ and }\dim F_S = \dim \Delta - \codim S + \ov{p}(S),$$
where $F_S$ is the smallest face of $\Delta$ containing  $s^{-1}(S)$.
The inclusions $\sigma^{-1}(S)\subset F_{S}\subset \nabla$ imply
$\sigma^{-1}(S)=s^{-1}(S)$. 
Thus, $F_{S}$ is also the smallest face of  $\nabla$ containing  $s^{-1}(S)$. From
$$\dim F_{S}=\dim\Delta-\codim S+\ov{p}(S)>\dim \nabla-\codim S+\ov{p}(S),$$
we deduce that the simplex $s$ is not $\ov{p}$-allowable.

\smallskip
(b) Let $\sigma''\colon \nabla \to X$ be the shared face of $\sigma$ and $\sigma'$.
Consider a stratum  $S \in \cS$ with $F_S =\hiru  T{\ov{p}}\sigma $. We have proven:
 $$F_S \subset \nabla \text{ and }
\emptyset \ne \sigma^{-1}(S) = {\sigma''}^{-1}(S)\subset \sigma'^{-1}(S).$$
Let $ F'_S$ be the smallest face of $\Delta$ containing $\sigma'^{-1}(S) $. 
It also contains  $T_{\ov{p}}(\sigma)$ and we have 
$$
\dim F'_S \leq \dim\Delta - \codim S + \ov{p}(S) = \dim T_{\ov{p}}(\sigma) ,
$$
since $\sigma'$ is $\ov{p}$-allowable. %
Thus we get  
$$T_{\ov{p}}(\sigma)  = F'_S \text{ and } \dim F'_S = \dim\Delta - \codim S + \ov{p}(S).$$
By definition of $\ov{p}$-bad faces, we conclude that $T_{\ov{p}}(\sigma') \subset F'_S$ and therefore
$T_{\ov{p}}(\sigma') \subset T_{\ov{p}}(\sigma)$. 
Since the simplices  $\sigma$ and $\sigma'$ play the same role we have $
T_{\ov{p}}(\sigma') =T_{\ov{p}}(\sigma)$. 
Property (a) gives that  $T_{\ov{p}}(\sigma)$ is a face of $\sigma''$.

\smallskip
(c) If the simplex $\sigma$ has a  $\ov{p}$-bad face $F$, then $F$ is included in the boundary of $\sigma$.
The result comes from (a).
\end{proof}

\begin{proof}[Proof of  \propref{prop:MVhomologie}]
Let us consider the two short exact sequences
\begin{equation}\label{equa:MVC}
\xymatrix@C=5mm{	
0 \ar[r] & C^{\ov{p}}_{*}(U\cap V)
\ar[r] & 
C^{\ov{p}}_{*}(U)\oplus C^{\ov{p}}_{*}(V)
\ar[r]^{\varphi} & 
 C^{\ov{p}}_{*}(U)+ C^{\ov{p}}_{*}(V)
 \ar[r] & 
 0\\
}\end{equation}
and
\begin{equation}\label{equa:MVgC}
\xymatrix@C=5mm{	
0 \ar[r] & \gC^{\ov{p}}_{*}(U\cap V)
\ar[r] & 
\gC^{\ov{p}}_{*}(U)\oplus \gC^{\ov{p}}_{*}(V)
\ar[r]^{\psi} & 
 \gC^{\ov{p}}_{*}(U)+ \gC^{\ov{p}}_{*}(V)
 \ar[r] & 
 0.\\
}\end{equation}
where the chain maps  $\varphi$ and $\psi$ are defined by  $(\alpha,\beta) \mapsto \alpha + \beta$.
The existence of the Mayer-Vietoris exact sequences come from the facts  
that the two inclusions $\im \varphi \hookrightarrow C^{\ov{p}}_{*}(X)$
and
$\im \psi \hookrightarrow \gC^{\ov{p}}_{*}(X)$
induce isomorphisms in homology.

The complex generated by the filtered simplices of $X$ is denoted by $C_{*}(X)$.
We have constructed two operators 
$\sd \colon C_{*}(X)\to C_{*}(X)$  and
$T\colon C_{*}(X) \to C_{*+1}(X)$ verifying $\sd \partial = \partial  \ \sd$ and 
$\partial T + T \partial  = \id -\sd$ (cf.  \cite[Lemma A.16]{CST1}).
These two operators are constructed as in the classical homology setting from the barycentric subdivision,
see \cite[Pages 176-179]{Span}. We denote by $\Delta^{per}(\Delta)$ the complex generated by the simplices
$\zeta\colon \Delta^{\ell}\to \Delta$ such that
$\zeta(\sum_{i}t_{i}e_{i})=\sum_{i}t_{i}\zeta(e_{i})$
and $\dim \zeta^{-1}(F) \leq \dim F$, for any face $F$ of $\Delta$. 
On such simplex, we consider the classical barycentric subdivision,
$\sd(\zeta)$, and set
$\sd(\sigma)=\sigma_{*}(\sd([\Delta])$
as in \cite{Span}. The operator $T$ is constructed similarly in \cite[Lemma A.16]{CST1}.
  We have also established the inequality
 \begin{equation}\label{equa:kfuche}
 ||\sigma \circ \zeta ||_{i} \leq 
||\sigma  ||_{i}.
\end{equation}
We decompose the proof of   \propref{prop:MVhomologie} in three steps.

\smallskip
$\bullet$ \emph{First step: 
The subdivision operator
$\sd \colon C^{\ov{p}}_{*}(X) \to C^{\ov{p}}_{*}(X)$  is homotopic to the identity and induces an operator, 
$\sda \colon  \gC^{\ov{p}}_{*}(X) \to  \gC^{\ov{p}}_{*}(X)$, also homotopic to the identity}.  
For the first assertion, it suffices to prove that the operators $\sd$ and $T$ preserve the perverse degree, that is, they induce the operators
$\sd \colon C^{\ov{p}}_{*}(X) \to C^{\ov{p}}_{*}(X)$  and 
$T \colon C^{\ov{p}}_{*}(X) \to C^{\ov{p}}_{*+1}(X)$.

Let us begin with $\sd$.  Given a filtered simplex $\sigma \colon \Delta \to X$,
 we need to prove the inequality $||\sigma \circ \zeta ||_S \leq ||\sigma||_S$,
for each   $S\in\cS$ and each simplex $\zeta \in \Delta^{per}(\Delta)$. 
Write $i = \codim S$. This inequality is clear when  $\im (\sigma\circ\zeta) \cap S =\emptyset$. 
If $\im (\sigma \circ \zeta)  \cap S \ne\emptyset$, then $\im \sigma   \cap S \ne\emptyset$ and we get
$$
||\sigma \circ \zeta ||_S = ||\sigma \circ \zeta ||_{i} \leq 
||\sigma  ||_{i} =  ||\sigma ||_S,
$$
where the inequality comes from (\ref{equa:kfuche}).  A similar proof works for the operator 
$T$.

Let us proceed now with the operator $\sda$. 
Let $\sigma \colon \Delta \to X$ be a filtered simplex and  $\zeta$ a simplex in $ \Delta^{per}(\Delta)$. 
 We consider the  $\sigma$-decomposition $\Delta = \Delta_0 * \cdots * \Delta_n$  
 and the $(\sigma \circ \zeta)$-decomposition  $\Delta = \Delta'_0 * \cdots * \Delta'_n$. 
 By definition of $\zeta$ we have $\Delta'_0 * \cdots * \Delta'_{n-1} \subset \Delta_0 * \cdots * \Delta_{n-1}$ 
 and therefore $\dim \Delta_n \leq  \dim \Delta'_n$. So, if $\sigma$ is a regular simplex then 
 all the simplices of $\sd (\sigma)$ are also regular simplices. On the other hand, 
 if $\sigma$ is not a regular simplex, that is $\im \sigma \subset \Sigma$,   then $\sigma \circ \zeta$ neither is. 
 From \defref{def:tameGreg}, we can write $\partial \sigma=\partial_{\sing}\sigma+ \gd\sigma$, with
 $\im(\partial_{\sing}\sigma)\subset \Sigma$. We decompose
 \begin{itemize}
 \item $\sd (\partial \sigma) =\sd (\partial_{\sing}\sigma) + \sd (\gd \sigma)$,
 \item $\partial (\sd \sigma)= \partial_{\sing}(\sd \sigma)+ \gd (\sd \sigma)$.
 \end{itemize}
 From the previous part, we know $\sd(\partial \sigma)=\partial(\sd \sigma)$. The equality of their regular parts
 gives $\sd (\gd \sigma)=\gd (\sd \sigma)$. Thus the map $\sd$ induces a chain map  
 $\sda \colon \gC^{\ov{p}}_{*}(X) \to \gC^{\ov{p}}_{*}(X)$.
A similar proof works for the construction of a homotopy operator $\gT\colon \gC_{*}(X) \to \gC_{*+1}(X)$ from $T$.

\smallskip
$\bullet$ \emph{Second step. We prove the two following implications:
\begin{equation}\label{equa:subdivisionetintersection}
\xi  \in C^{\ov{p}}_{*}(X)\Rightarrow  \sd^k \xi \in 
C^{\ov{p}}_{*}(U)+ C^{\ov{p}}_{*}(V) \text{ for some } k \geq 0,
\end{equation}
and
\begin{equation}\label{equa:subdivisionetintersectionBis}
\xi  \in \gC^{\ov{p}}_{*}(X)\Rightarrow  \sda^k \xi \in 
\gC^{\ov{p}}_{*}(U)+ \gC^{\ov{p}}_{*}(V) \text{ for some } k \geq 0.
\end{equation}
}
The \emph{canonical decomposition} of a $\ov p$-allowable chain $\xi$ is $\xi = \xi_0 + \sum_{\tau \in I_\xi} \xi_\tau$ where:
\begin{itemize}
\item[-] $\xi_0$ is the chain containing the simplices of $\xi$ without $\ov p$-bad faces.
\item[-] $I_\xi$ is the family of the $\ov p$-bad faces of the simplices of $\xi$.
\item[-] $\xi_\tau$ is the chain containing the simplices of $\xi$ having $\tau$ as $\ov p$-bad face.
\end{itemize}
The boundaries $\partial \xi_0$ and $\gd \xi_0$  are $\ov p$-allowable chains.
A non-$\ov p$-allowable face (resp.  regular non-$\ov p$-allowable face) $s$ of a simplex $\sigma$ of $\xi_\tau$ contains necessarily $\tau$. 
When $\partial \xi$  (resp. $\gd \xi$) is a $\ov p$-allowable chain then $s$ does not appear in 
$\partial \xi$ (resp. $\gd \xi$). 
So, there exists another simplex $\sigma'$ in $\xi$ having $s$ as a face. 
Since $s$ contains $\tau$ then $\tau$ is the bad face of $\sigma'$. 
We conclude that 
 $\partial \xi_\tau$ (resp. $\gd \xi_\tau$) is also a $\ov p$-allowable chain.
 (These facts come from \propref{prop:badface}.) For each $\tau \in I_{\xi}$, we have proven
 $$
\left( 
\xi\in C^{\ov{p}}_{*}(X)\Rightarrow  
 \xi_0, \xi_\tau \in C^{\ov{p}}_{*}(X)
 \right)
 \text{ and }
 \left(
  \xi\in \gC^{\ov{p}}_{*}(X)\Rightarrow  
 \xi_0, \xi_\tau \in \gC^{\ov{p}}_{*}(X)
 \right).
 $$
The usual subdivision argument gives the existence of an integer $k\geq 1$ such that 
the canonical decomposition of $\sd^k \xi$ verifies the following properties.
\begin{itemize}
\item[-] Each simplex of $(\sd^k \xi)_0$ lives in $U$ or in $V$.
\item[-] For each $\tau \in I_{\sd^k \xi}$ the chain $(\sd^k \xi)_\tau$ lives in $U$ or in $V$.
\end{itemize}
This gives \eqref{equa:subdivisionetintersection} and \eqref{equa:subdivisionetintersectionBis}

\smallskip
$\bullet$ \emph{Third step. 
The two inclusions $\iota\colon \im \varphi \hookrightarrow C^{\ov{p}}_{*}(X)$
and
$\kappa\colon \im \psi \hookrightarrow \gC^{\ov{p}}_{*}(X)$
induce isomorphisms in homology.}

Let $[\xi]\in H^{\ov{p}}_{*}(X)$. 
The first step implies   $[\xi]=[\sd^i\xi]$, for each  $i\geq 0$.
From the second step, we get an integer  $k\geq 0$ with  $\sd^k\xi\in\im\varphi$. 
This gives the surjectivity of  $\iota_{*}$.

In order to prove the injectivity of $\iota_{*}$,
we consider  $[\alpha]\in  H_{*}(\im\varphi)$
and  $\xi\in C^{\ov{p}}_{*+1}(X)$  with 
$\alpha=\partial\xi$. The second step gives an integer  
$k\geq 0$ with  $\sd^k\xi\in\im\varphi$, and therefore we have
$[\alpha]=[\sd^k\alpha]=[\sd^k(\partial\xi)]=[\partial (\sd^k\xi)]$ 
and  $[\alpha]=0$ in   $H_{*}(\im\varphi)$.

The same arguments work for $\kappa_{*}$ and the connecting map is computed in the usual way.
\end{proof}

Excision property also exists in this context.

\begin{definition}\label{prop:homorelative}
Let  $(X,\ov{p})$ be a perverse space and $U \subset X$ an open subset.
The \emph{relative chain complexes} are the quotients
$C^{\ov{p}}_{*}(X,U;G)=C^{\ov{p}}_{*}(X;G)/ C^{\ov{p}}_{*}(U;G)$
and
$\gC^{\ov{p}}_{*}(X,U;G)=\gC^{\ov{p}}_{*}(X;G)/ \gC^{\ov{p}}_{*}(U;G)$.
The \emph{relative homologies},
$H^{\ov{p}}_{*}(X,U;G)$ and $\gH^{\ov{p}}_{*}(X,U;G)$,  are the homologies of these two complexes, respectively.
\end{definition}

These relative homologies give long exact sequences:
\begin{equation}\label{equa:suiterelative}
\dots\to 
H^{\ov{p}}_{i}(U;G)\to
 H^{\ov{p}}_{i}(X;G)\to
 H^{\ov{p}}_{i}(X,U;G)\to
 H^{\ov{p}}_{i}(U;G)\to
\dots
\end{equation}
and
\begin{equation}\label{equa:suiterelativeg}
\dots\to 
\gH^{\ov{p}}_{i}(U;G)\to
 \gH^{\ov{p}}_{i}(X;G)\to
 \gH^{\ov{p}}_{i}(X,U;G)\to
 \gH^{\ov{p}}_{i}(U;G)\to
\dots
\end{equation}

Proceeding as in the classical case 
 (\cite[Chapter 4, Section 6, Corollary 5]{Span}), we get an excision property.
 
\begin{corollaire}\label{prop:Excisionhomologie}
Let $(X, \ov{p})$ be a perverse space. If $F$ is a closed subset of  $X$ and $U$ is an open subset of $X$ 
with $F\subset U$,
then the natural inclusion  $(X\backslash F,U\backslash F)\hookrightarrow (X,U)$ induces the isomorphisms
$H^{\ov{p}}_{i}(X\backslash F,U\backslash F ;G) \cong H^{\ov{p}}_{i}(X,U ;G) \ \hbox{ and } \ 
H^{\ov{p}}_{i}(X\backslash F,U\backslash F ;G)\cong H^{\ov{p}}_{i}(X,U ;G) .$
\end{corollaire}

The principle of $\cU$-small chains also exists in this context.

\begin{corollaire}\label{prop:chaineUpetite}
Let  $(X, \ov{p})$ be a perverse space endowed with an open covering $\cU$. 
Denote $C^{\ov{p},\cU}_{*}(X;G)$ 
(resp. $\gC^{\ov{p},\cU}_{*}(X;G)$) 
the sub-complex of  $C^{\ov{p}}_{*}(X;G)$ (resp. $\gC^{\ov{p}}_{*}(X;G)$) 
made up of chains whose support is included in an element of  $\cU$.
Then,  the canonical inclusions,
$C^{\ov{p},\cU}_{*}(X;G)\hookrightarrow C^{\ov{p}}_{*}(X;G)$ 
and
$\gC^{\ov{p},\cU}_{*}(X;G)\hookrightarrow \gC^{\ov{p}}_{*}(X;G)$),
induce  isomorphisms in homology.
\end{corollaire}

\begin{proof}
By construction, for any element 
 $\xi \in C^{\ov{p}}_{*}(X)$ (resp. $\gC^{\ov{p}}_{*}(X)$)  
 there exists an integer $m$ with 
$\sd^m \xi \in  C^{\ov{p},\cU}_{*}(X)$ (resp. $\sda^m \xi \in \gC^{\ov{p},\cU}_{*}(X)$).
Moreover $\sd^m\xi $ and $\sda^m\xi $ are cycles if $\xi $ is one.
The  argument used in the third step of the proof of   \propref{prop:MVhomologie} gives the injectivity and the surjectivity of 
the two canonical inclusions.
\end{proof}

The following result is used in  the proof of \propref{prop:local}. 

\begin{corollaire}\label{cor:SfoisX}
Let $(X,\ov{p})$ be a perverse space. Let $M$ be a manifold which is an $\ell$-homological sphere. 
Then, there are isomorphisms,
$H_{k}^{\ov{p}}(M\times X;G)\cong H_{k}^{\ov{p}}(X;G)\oplus H_{k-\ell}^{\ov{p}}(X;G)$
and
$\gH_{k}^{\ov{p}}(M\times X;G)\cong \gH_{k}^{\ov{p}}(X;G)\oplus \gH_{k-\ell}^{\ov{p}}(X;G).$
\end{corollaire}

\begin{proof}
This is a consequence of Theorems~\ref{thm:martinmacphersonhomologie}, 
\ref{thm:martingreghomologie} and 
\cite[Theorems 5.2.25 and 6.3.20]{FriedmanBook}.
\end{proof}

\section{Proofs of Theorems \ref{thm:martinmacphersonhomologie} and \ref{thm:martingreghomologie}}\label{sec:thmA}

The method of  proof is a variant of \cite[Theorem 10]{MR800845}, 
\cite[Lemma 1.4.1]{MR2210257}. We choose the formulation presented by Friedman in 
\cite[Section 5.1]{FriedmanBook}. %

\begin{theoremv}{\rm(\cite[Theorem 5.1.1]{FriedmanBook})}\label{thm:gregtransformationnaturelle}
Let $\cF_{X}$ be the category whose objects are (stratified homeomorphic to) open subsets
of a given  CS set $X$ and whose morphisms are  stratified homeomorphisms and inclusions.
Let  $\cAb_{*}$ be the category of graded abelian groups. Let $F^{*},\,G^{*}\colon \cF_{X}\to \cAb$
be two functors and
 $\Phi\colon F_{*}\to G_{*}$ be a natural transformation satisfying
 the conditions listed
below.
\begin{enumerate}[(i)]
\item $F^{*}$ and $G^{*}$ admit exact Mayer-Vietoris sequences and the natural transformation $\Phi$ 
 induces a commutative diagram between these sequences.
\item If $\{U_{\alpha}\}$ is an increasing collection of open subsets of $X$  and 
$\Phi\colon F_{*}(U_{\alpha})\to G_{*}(U_{\alpha})$ is an isomorphism 
for each $\alpha$, then $\Phi\colon F^{*}(\cup_{\alpha}U_{\alpha})\to G^{*}(\cup_{\alpha}U_{\alpha})$  is an isomorphism.
\item If $L$ is a compact filtered space such that 
$X$  has an open subset  which is stratified homeomorphic
to $\R^i\times \rc L$ and, if
$\Phi\colon F^{*}(\R^i\times (\rc L\backslash \{\tv\}))\to G^{*}(\R^i\times (\rc L\backslash \{\tv\}))$
is an isomorphism, then so is
$\Phi\colon F^{*}(\R^i\times \rc L)\to G^{*}(\R^i\times \rc L)$. Here, $\tv$ is the apex of the cone $\rc L$.

\item If $U$ is an open subset of X contained within a single stratum and homeomorphic
to an Euclidean space, then $\Phi\colon F^{*}(U)\to G^{*}(U)$ is an isomorphism.
\end{enumerate}
Then $\Phi\colon F^{*}(X)\to G^{*}(X)$ is an isomorphism.
\end{theoremv}

Before using \thmref{thm:gregtransformationnaturelle}, we need to compute the (tame) intersection homology of a cone. 
Let us emphasize the difference between the two homologies.

\begin{proposition}[Homology of a cone]\label{prop:homologiecone}
Let $X$  be a compact filtered space of dimension $n$.
Consider a perversity $\ov p$ on the cone $\rc X$.
We have:
$$
H^{\ov{p}}_{k}(\rc X;G)
\cong
\left\{
\begin{array}{ll}
H^{\ov{p}}_{k}( X;G)& \hbox{if } k< n-\ov{p}( \{ \tv \}),\\
0 & \hbox{if } 0\ne k \geq n-\ov{p}( \{ \tv \}),\\
G& \hbox{if } 0= k \geq n-\ov{p}( \{ \tv \}),
\end{array}
\right.
$$
and
$$
\gH^{\ov{p}}_{k}(\rc X;G)
\cong
\left\{
\begin{array}{ll}
\gH^{\ov{p}}_{k}( X;G)& \hbox{if } k< n-\ov{p}( \{ \tv \}),\\
0 & \hbox{if }  k \geq n-\ov{p}( \{ \tv \}).
\end{array}
\right.
$$
In the two cases, the isomorphism at the top is induced by the map  $\iota_{\rc X}\colon X\to \rc X$, $x\mapsto [x,1/2]$.
The other non-zero determination comes from the canonical inclusion
$C^{\ov{p}}_{k}(\rc X;G)\to C_{k}(\rc X;G)$.
\end{proposition}

\begin{proof}
We proceed in several steps. 

$\bullet$ \emph{In the low degrees.} 
Let $k=\dim\Delta\leq n-\ov{p}( \{ \tv \})$ and
 $\sigma\colon \Delta\to \rc X$  be a   $\ov{p}$-allowable filtered simplex.
The allowability condition $\|\sigma\|_{\{\tv\}}\leq k-(n+1)+\ov{p}( \{ \tv \})<0,$ implies 
 $\sigma^{-1}(\{\tv\})=\emptyset$.
 We therefore have
 $C^{\ov{p}}_{\leq n-\ov{p}(\tv)}(\rc X)=C^{\ov{p}}_{\leq n-\ov{p}(\tv)}( X\times ]0,1[)
 $
 and
 $\gC^{\ov{p}}_{\leq n-\ov{p}(\tv)}(\rc X)=\gC^{\ov{p}}_{\leq n-\ov{p}(\tv)}( X\times ]0,1[)
 $.
 Thus, \corref{cor:RfoisX} gives the requested isomorphisms.

$\bullet$ \emph{Construction of a cone operator}.
If $[x,t]\in\rc X$, we define
$s\cdot [x,t]=[x,st]$, for each $s\in[0,1]$. 
Let $\sigma\colon \Delta\to \rc X$  be a   $\ov{p}$-allowable filtered simplex.
We define  the cone $c\sigma\colon \{\tv\}\ast \Delta\to \rc X$
by
$c\sigma(sx+(1-s)\tv)=s\cdot\sigma(x)$, for each $s\in[0,1]$ and $x\in\Delta$.

If $\dim\Delta=k\geq n-\ov{p}(\{\tv\})$, we prove that $c\sigma$ is $\ov{p}$-allowable. For that, let us observe that
the cone $\rc X$ has strata of two types: the products $S\times ]0,1[$ with  $S \in \cS_{X}$ 
and the apex $\{\tv\}$ of the cone. 

\begin{enumerate}[(i)]
\item  For a stratum $S\times ]0,1[$, we have
$$
(c\sigma)^{-1}(S \times ]0,1[)=\{sa+(1-s)\tv\mid s\cdot \sigma(a)\in S \times ]0,1[\}
=\sigma^{-1}(S \times ]0,1[)\times ]0,1[.
$$
If $(c\sigma)^{-1}(S \times ]0,1[)\neq\emptyset$,
then $\sigma^{-1}(S \times ]0,1[)\neq\emptyset$ and therefore
\begin{eqnarray*}
||c\sigma||_{S \times ]0,1[}&=&
 1 + ||\sigma||_{S \times ]0,1[}\\
&\leq &
1 + k-\codim (S\times ]0,1[) + \ov{p}(S\times ]0,1[)\\
& \leq&
\dim (\{\tv\}\ast\Delta)-\codim (S\times ]0,1[) + \ov{p}(S\times ]0,1[).
\end{eqnarray*}
\item For the apex, we have
$(c\sigma)^{-1}(\{ \tv\})=\{\tv\}\ast \sigma^{-1}(\{ \tv\})$
which implies
\begin{eqnarray*}
||c\sigma||_{\{\tv\}}&=& 
\left\{ \begin{array}{ll}
0 & \hbox{if }  \sigma^{-1}(\{\tv\})= \emptyset,\\
1 + ||\sigma||_{\{\tv\}} & \hbox{if }  \sigma^{-1}(\{\tv\})\neq \emptyset.
\end{array}
\right. 
\end{eqnarray*}
 If $\sigma^{-1}(\{\tv\})\neq \emptyset$, we get
 $\|c\sigma\|_{\{\tv\}}= 1+\|\sigma\|_{\{\tv\}}\leq \dim (\{\tv\}\ast \Delta)- \codim \{\tv\} + \ov{p}( \{ \tv \}).$
 If  $\sigma^{-1}(\{\tv\})= \emptyset$, the  $\ov{p}$-allowability condition relatively to  $\{\tv\}$ becomes
 $$0=\|c\sigma\|_{\{\tv\}}\leq \dim (\tv\ast \Delta)-\codim \{\tv\}+\ov{p}(\{\tv\})=(k+1)-(n+1)+\ov{p}(\{\tv\}),$$
which is exactly, $k\geq n-\ov{p}( \{ \tv \})$.
\end{enumerate}
We have proven that $c\sigma$ is a  $\ov p$-allowable filtered simplex, as required.

$\bullet$ \emph{Let $0\ne k \geq n-\ov{p}( \{ \tv \})$}.
Given a cycle $\xi \in C^{\ov{p}}_{k}(X)$ (resp. $\xi' \in \gC^{\ov{p}}_{k}(X)$), 
the equality $\partial c \xi = \xi$ (resp. $\gd c \xi' =\xi'$) gives the result.

$\bullet$ \emph{Let $0 =k \geq n-\ov{p}( \{ \tv \})$}.
Given a simplex  $\sigma \in C^{\ov{p}}_{0}(X)$ (resp. $\sigma' \in \gC^{\ov{p}}_{0}(X)$),
 the equality $\partial c \sigma = \sigma - \tv$ (resp. $\gd c\sigma' =\sigma'$) gives the result.
\end{proof}

\propref{prop:homologiecone},  \corref{cor:RfoisX} and the long exact sequence of a pair imply the following result.

\begin{corollaire}\label{cor:homologieconerel}
With the hypotheses and  the notations of \propref{prop:homologiecone}, we have
$$
H^{\ov{p}}_{k}(\rc X,\rc X \backslash \{ \tv\} ;G)
=\left\{
\begin{array}{cl}
\widetilde{H}^{\ov{p}}_{k-1}(\rc X;G)
&\text{ if } k\geq n+1-\ov{p}( \{ \tv \})
,\\[.2cm]
0&\text{ if }  k \leq  n-\ov{p}( \{ \tv \}),\\[.2cm]
\end{array}\right.$$
where $\widetilde{H}$ is the reduced homology, and
$$
\gH^{\ov{p}}_{k}(\rc X,\rc X \backslash \{ \tv\} ;G)
=\left\{
\begin{array}{cl}
\gH^{\ov{p}}_{k-1}(\rc X;G)
&\text{ if } k\geq n+1-\ov{p}( \{ \tv \})
,\\[.2cm]
0&\text{ if }  k \leq  n-\ov{p}( \{ \tv \}).
\end{array}\right.$$
\end{corollaire}

\begin{proof} [Proof of  \thmref{thm:martinmacphersonhomologie}]
We verify the conditions of  \thmref{thm:gregtransformationnaturelle} 
for the natural transformation  
$\Phi\colon H^{\ov{p}}_{*}(U)\to I^{\ov{p}}H_{*}(U)$
induced by the canonical inclusion
$C^{\ov{p}}_{*}(U)\hookrightarrow I^{\ov{p}}C_{*}(U)$.

\medskip
(a) The Mayer-Vietoris exact sequences have been constructed in \propref{prop:MVhomologie}
for the  complex $C^{\ov{p}}_{*}(X)$
and in \cite[Theorem 4.4.19]{FriedmanBook}  for the complex $I^{\ov{p}}C_{*}(X)$.

\medskip
(b) This a classical argument for  homology theories with compact supports.

\medskip
(c) Let $L$ be a compact filtered space such that the natural inclusion induces the isomorphism
$$\Phi_{(\R^i\times (\rc L\backslash \{\tv\})}\colon
H_{*}^{\ov{p}}(\R^i\times (\rc L\backslash\{\tv\}))
\xrightarrow[]{\cong} 
I^{\ov{p}}H_{*}(\R^i\times (\rc L\backslash\{\tv\})).$$
Since  $\R^i\times ]0,1[\times L =  \R^i\times (\rc L\backslash \{\tv\})$, we get the isomorphism
$$\Phi_{\R^i\times ]0,1[\times L}\colon 
H_{*}^{\ov{p}}(\R^i\times ]0,1[ \times L)
\xrightarrow[]{\cong} 
I^{\ov{p}}H_{*}(\R^i\times ]0,1[ \times L).$$
Let us consider the following commutative diagram
$$\xymatrix{
H_{*}^{\ov{p}}(\R^i\times ]0,1[ \times L)
\ar[d]_{\pr_{*}}\ar[rr]^-{\Phi_{\R^i\times ]0,1[\times L}}
&&
I^{\ov{p}}H_{*}(\R^i\times ]0,1[ \times L)
\ar[d]^{\pr_{*}}\\
H_{*}^{\ov{p}}(L)
\ar[d]_{(\iota_{\rc L})_{*}}\ar[rr]^-{\Phi_{L}}
&&
I^{\ov{p}} H_{*}(L)
\ar[d]^{(\iota_{\rc L})_{*}}\\
H_{*}^{\ov{p}}(\rc L)
\ar[rr]^-{\Phi_{\rc L}}
&&
I^{\ov{p}} H_{*}(\rc L).
}$$
From \corref{cor:RfoisX} and \cite[Example 4.1.13.]{FriedmanBook}, we know that
 the two maps $\pr_{*}$, induced by the canonical projections, are isomorphisms. 
 We conclude that  $\Phi_{L}$ is an isomorphism.

If $*< n-\ov{p}( \{ \tv \})$ then \propref{prop:homologiecone} and 
\cite[Theorem 4.2.1]{FriedmanBook} imply that the two maps 
$(\iota_{\rc L})_*$ are isomorphisms. So, $\Phi_{\rc L}$ is an isomorphism in these degrees.

When $*\geq n-\ov{p}( \{ \tv \})$, the map $\Phi_{\rc L}$ is directly an isomorphism (cf. \propref{prop:homologiecone} and \cite[Section 5.4]{FriedmanBook}).

\medskip
(d) The map $\Phi\colon \lau H {\ov{p}}*{U;G}\to \bost I  {\ov{p}}H{*}{U;G}$ becomes the identity $G \to G$.
\end{proof}

\begin{proof} [Proof of  \thmref{thm:martingreghomologie}]
The same proof works, changing   \cite[Theorem 4.4.19]{FriedmanBook},  \cite[Example 4.1.13]{FriedmanBook} 
and  \cite[Theorem 4.2.1]{FriedmanBook}  by  
\cite[Theorem 6.3.12]{FriedmanBook},  \cite[Corollary 6.3.8]{FriedmanBook} 
and \cite[Theorem 6.2.13]{FriedmanBook} respectively  .
\end{proof}

We end this Section with concrete computations for some particular perversities.

\begin{proposition}\label{pro:PerDir}
For any  perverse CS set  $(X,\ov p)$, we have the following equalities. 

\begin{enumerate}[\rm a)]
\item $H^{\ov{t}}_{*}(X;G) = \gH^{\ov{t}}_{*}(X;G)   =H_{*}(X;G)$, if $X$ is normal,
\item $H^{\ov{p}}_{*}(X;G)  = H_{*}(X;G)$, if $\ov p > \ov t$,
\item $\gH^{\ov{p}}_{*}(X;G) = H_{*}(X,\Sigma;G)$, if $\ov p > \ov t$,
\item $H^{\ov{p}}_{*}(X;G) = \gH^{\ov{p}}_{*}(X;G)   =H_{*}(X\backslash \Sigma;G)$, if $\ov p <\ov 0$.
\end{enumerate}
\end{proposition}

\begin{proof}
The argument is similar in the four cases, by applying \thmref{thm:gregtransformationnaturelle}. 
We only detail the proof of property a). In fact, the only relevant point is the item (iii).

Let $U \subset X$ be an open subset.
We use \remref{igual} and  consider the natural transformation  
 $\Phi_{U}\colon H_{*}^{\ov{t}}(U)\to  H_{*}(U)$, 
 induced by the canonical inclusion
$C_{*}^{\ov{t}}(U)\hookrightarrow C_{*}(U)$.
In order to verify Property (iii) of \thmref{thm:gregtransformationnaturelle},
we consider a compact stratified space $L$ and the commutative diagram
$$\xymatrix{
H^{\ov{t}}_*(L)
\ar[d]_{(\iota_{\rc L})_{*}}\ar[rr]^-{\Phi_{L}}
&&
H_{*}( L)
\ar[d]^{(\iota_{\rc L})_{*}}\\
H^{\ov{t}}_*(\rc L)
\ar[rr]^-{\Phi_{\rc L}}
&&
H_{*}(\rc L).
}$$
where $\Phi_L$ is an isomorphism.
We have to prove that $\Phi_{\rc L}$ is an isomorphism.
Since $n-\ov{t}(\tv) = n-(n-1) = 1$,  
\propref{prop:homologiecone} implies that
$
H^{\ov{t}}_*(\rc L)= H^{\ov{t}}_0(\rc L) =
H^{\ov{t}}_0(L)$.
From an induction on the depth of $X$, we get 
$H^{\ov{t}}_0(L)= H_0(L)$. 
Since the space $X$ is normal, then  $L$ is connected and $\Phi_{\rc L}$ is an isomorphism. 
\end{proof}

Notice that the item c) implies that  tame intersection homology is not a topological invariant.

\section{Topological invariance. \thmref{thm:invariancegenerale}}\label{sec:invariancel}

\begin{quote}
In this section, we present the construction of the intrinsic CS set associated to a CS set.
In \thmref{thm:invariancegenerale}, we use it to prove
a topological invariance of intersection homology for general perversities verifying the conditions of \defref{def:llperversite}.
A consequence is the topological invariance for general perversities  
obtained as pull-back's of a perversity defined on the intrinsic CS set.
The situation of tame intersection homology is also described.
\end{quote}

The intrinsic CS set associated to a CS set is defined by  King in \cite{MR800845}, crediting the idea to Sullivan.
A careful study of this notion has been done by  Friedman in \cite[Section 2.10]{FriedmanBook}.

\begin{definition} \label{def:RelEquiv1}
Two points $x_{0}$, $x_{1}$ of a topological space $X$ are  \emph{equivalent} if there exists a homeomorphism $ (U_{0},x_{0})\cong (U_{1},x_{1})$
between two neighborhoods of  $x_{0}$ and $x_{1}$.
This equivalence relation is denoted by $x_{0}\sim x_{1}$.
\end{definition}

Two points belonging to the same stratum of a CS set are equivalent.
So the equivalence classes are union of strata.
The  construction of the intrinsic CS set is as follows.

\begin{proposition}{\cite[Section 2.10]{FriedmanBook}}\label{prop:xetxstar}
Let $X$ be a  CS set. We denote by  $X^*_{i}$ the union of equivalence classes, relatively to $\sim$,
made up of strata of $X$ whose dimension is lower than or equal to $i$.
Then, the topological space $X$ endowed with this filtration is a CS set. 
Moreover, this filtration  does not depend on the initial CS structure considered on $X$.
\end{proposition}

\begin{definition}\label{def:Xstar}
With the notations of \propref{prop:xetxstar}, 
the space $X^*$ is called the
\emph{intrinsic CS set} associated to the CS set $X$ and its filtration 
is the \emph{intrinsic filtration} of $X$.
 The identity map $\nu\colon X\to X^*$ is the \emph{ intrinsic aggregation.}
 For sake of simplicity, we denote $\cS^*=\cS_{X^*}$.
\end{definition}

We need to relate the strata of $X$ and $X^*$ in order to compare their homologies. 
This is the goal of the following definition.
\begin{definition}\label{def:source}
Let  $X$ be a  CS set with intrinsic aggregation  $\nu\colon X\to X^*$. 
A stratum $S$ of $X$  is a \emph{source of the stratum}
$T \in \cS^*$ if $\nu(S) \subset T$ and $\dim S = \dim T$.
\end{definition} 

\begin{proposition}\label{prop:source}
Let  $X$ be a  CS set with intrinsic aggregation  $\nu\colon X\to X^*$. Then, the following properties hold.
\begin{enumerate}[\rm (a)]
\item The intrinsic aggregation is a stratified map.
\item Let $S\in\cS_{X}$ be a singular stratum. Then there exists a source stratum, $R$, of $S^{\nu}$ such that
$S\preceq R$ and $R\sim S$.
\item The  intrinsic filtration of an open subset $U \subset X$ is  
$U_{i}^*=X^*_{i}\cap U$. Let 
$\iota\colon U\to X$, $j\colon U^*\to X^*$  be the canonical injections 
and  $\nu^U\colon U\to U^*$ the intrinsic aggregation of $U$. 
Then, $\nu \circ \iota = j \circ \nu^U$.
\end{enumerate}
\end{proposition}

\begin{proof}
Let $T$ be a stratum of $X^*$. The set
$\nu^{-1}(T)=\cup_{\nu(S)\subset T} S$
is a topological manifold of the same dimension than $T$.
Moreover, the family of strata,
$\{S\mid \nu(S)\subset T\}$,
is a locally finite family of sub-manifolds of $T$.

(a) From $\nu(S)\subset T$, we deduce $\dim S\leq \dim T$ and $\codim T\leq \codim S$.
Therefore, $\nu$ is a stratified map.

(b) It is equivalent to prove that the union of the sources of a given stratum $T \in \cS^*$ is a dense subset of $T$.
 The union of the sources of $T$ being
$\{S\mid \nu(S)\subset T \text{ and } \dim S=\dim T\}$,
the density follows from dimension reasons.

(c) This is true by definition, see \cite[Lemma 2.2.10]{FriedmanBook} for more details.
\end{proof}

\begin{definition}\label{def:RelEquiv2}
Let $X$ be a CS set. Two strata $S,S'\in \cS$ are  \emph{equivalent} 
if there exist $x_{0}\in S$ and $x_{1}\in S'$ with $x_{0}\sim x_{1}$. This equivalence relation is denoted by $S\sim S'$.
\end{definition}

\begin{remarque}\label{poke}
The original topological invariance of the intersection homology was established by  Goresky and  MacPherson 
in \cite{GM1}. 
King proposed a different approach in \cite{MR800845} by proving that the intrinsic aggregation 
$\nu \colon \ X \to X^*$ induces an isomorphism in intersection homology (see also \cite[Section 5.5.3]{FriedmanBook}). 
King did it with slight more general perversities than GM-perversities. 
He was using loose perversities $\ov{p}$ verifying 
\begin{enumerate}[(i)]
\item $\ov p \geq \ov 0$,
\item $\ov p( k) \leq \ov p(k+1) \leq \ov p(k) +1 $ if $k\geq 1$.
\end{enumerate}
We call them \emph{King positive perversities.}
If we write them by using strata, the conditions
(i) and (ii) are equivalent to 
\begin{enumerate}[(i')]
\item $\ov p (S) \geq 0$ if $S \in \cS$,
\item $\codim S' \leq \codim S \Longrightarrow \ov p( S' ) \leq \ov p(S) \leq \ov p(S') +\codim S -\codim S' $ 
for each singular strata $S,S'$. 
\end{enumerate}
Our next step is to weaken as much as possible these two conditions in order to obtain 
the broadest family of loose perversities where topological invariance holds. 
These are the K-perversities we define below.

\smallskip
Let us also observe that, in the case of a  loose perversity, the same perversity is defined on $X$ and $X^*$. 
At the opposite,
in our setting a perversity is defined on the set of strata of  $X$. Thus, the first point is 
``how can we construct a perversity on $X^*$ from a perversity on $X$?''
This the heuristic justification of the  condition (C) in the next definition.
\end{remarque}

\begin{definition}\label{def:llperversite}
A \emph{K-perversity on a CS set} $X$ is a perversity   $\ov{p}$ verifying the following properties.

(A)  If $S \sim R$ are two strata of $\cS$, with $S$ singular and $R$ regular, then  $0\leq \ov{p}(S)$.

(B) If $S'$ is singular and a source  stratum then
$$S \preceq S' \text{ and } S\sim S'
\Longrightarrow \ov{p}(S') \leq \ov{p}(S) \text{ and } D\ov{p}(S')\leq D\ov{p}(S).
$$

(C) If  $S ,S' \in \cS$ are two source strata of the same stratum then  
$\ov{p}(S)=\ov{p}(S').$

\medskip\noindent
Moreover, a \emph{$K^*$-perversity} is a $K$-perversity verifying the following condition.

(D) If $S \sim R$ are two strata of $\cS$, with $S$ singular and $R$ regular, then  $0\leq D\ov{p}(S)$.

\end{definition}

The necessity of these conditions are studied in \remref{rem:yes}. We begin with a
 list of some basic observations related to this definition.
\begin{remarque} \label{69}
$\bullet$ A  King positive perversity is a K-perversity. 

$\bullet$ Condition (D) is 
satisfied if there is no singular stratum of $X$ becoming regular in $X^*$.

$\bullet$ A perversity $\ov p$ is a K$^*$-perversity if, and only if, $\ov p$ and $D\ov p$ are $K$-perversities.
In fact, conditions (A) and (D) imply $0 \leq \ov p(S) \leq \codim S -2$, when $S \sim R$ are two strata of $\cS$
 with $S$ singular and $R$ regular. In this case, we have $\codim S \neq 1$.

$\bullet$ Any perversity $\ov p$ on the intrinsic CS set  $X^*$ is a  K-perversity. 
The pull-back perversity $\nu^*\ov p$ is also a  K-perversity.

$\bullet$ If $\ov{p}$ is a loose perversity, we have the folloxing equivalences,
\begin{eqnarray*}
\ov{p}(k+1)\leq \ov{p}(k)+1
&\Longleftrightarrow&
k-\ov{p}(k)-2\leq k+1-\ov{p}(k+1)-2\\
&\Longleftrightarrow&
D\ov{p}(k)\leq D\ov{p}(k+1).
\end{eqnarray*}
Thus condition (B) expressed in terms of codimension coincides with  (\ref{equa:king}).
\end{remarque}

We show how to construct a perversity on $X^*$ from a K-perversity on $X$.
\begin{proposition}\label{prop:source2}
Let  $X$ be a  CS set with intrinsic aggregation  $\nu\colon X\to X^*$ and $\ov{p}$ be  a K-perversity on $X$.
Then the following properties are satisfied.
\begin{enumerate}[\rm (a)]
\item The  map  $\nu_{*}\ov{p}$ given by
$$\nu_{*}\ov{p}(T)=\ov{p}(S),$$
where $S\in\cS_{X}$ is a source of the stratum $T \in\cS_{X^*}$, is a perversity on $X^*$.
\item Let  $U \subset X$ be an open subset.
Denote by 
$\iota\colon U\to X$, $j\colon U^*\to X^*$ the canonical injections and
$\nu^U\colon U\to U^*$ the intrinsic aggregation of $U$.
 Then, we have
$$\nu^U_{*}\,\iota^*\,\ov{p}=
j^*\,\nu_{*}\,\ov{p}.$$
\item The intrinsic aggregation induces a 
  chain map 
  $\nu_{*}\colon C^{\ov{p}}_{*}(X;G)\to C^{\nu_{*}\ov{p}}_{*}(X^*;G)$.
Moreover, if $\ov p$ is a K$^*$-perversity,  then it induces a chain map
 $\nu_{*}\colon \gC^{\ov{p}}_{*}(X;G)\to \gC^{\nu_{*}\ov{p}}_{*}(X^*;G)$.
\end{enumerate}
\end{proposition}

\begin{proof}
(a) The equality  $\nu_{*}\ov{p}(T)=\ov{p}(S)$ makes sense since any stratum  $T \in\cS^*$ has a source,
see \propref{prop:source}(b). This definition does not depend on the choice of $S$ as condition (C) implies.

(b) 
Any stratum of $U^*$ is a connected component,  ${}_{cc}(T \cap U^*)$, of the intersection of 
$U$ with a stratum $T$ of $X^*$.
A source of ${}_{cc}(T \cap U^*)$ is a connected component,  ${}_{cc}(S\cap U)$,
 of the intersection of  $U$ with a source $S$  of $T$. So, from item (a) and \defref{def:perversiteimagereciproque}, we get
$$\left\{
\begin{array}{l}
(j^*\nu_{*}\ov{p})( {}_{cc}(T \cap U^*) )
=
\nu_{*}\ov{p}( T)
=
\ov{p}(S),
\\
(\nu^U_{*}\iota^*\ov{p})( {}_{cc}(T \cap U^*) )
=
\iota^*\ov{p}( {}_{cc}(S \cap U) ) 
=
\ov{p}(S),
\end{array}\right.
$$
and  $\nu^U_{*}\,\iota^*\,\ov{p}=
j^*\,\nu_{*}\,\ov{p}.$

(c) 
Since $\nu$ is a stratified map (cf. \propref{prop:source}(a)), we can apply   \propref{prop:stratifieethomotopie}.
Let $S\in\cS_{X}$ be a singular stratum and set $T=S^{\nu}$. From \propref{prop:source}(b),
there exists a source stratum $R\in\cS_{X}$ of $T$ such that $R\preceq S$ and $R\sim S$.
From \defref{def:perversiteimagereciproque}, the definition of $\nu_{*}\ov{p}$ and Property $(B)$ of 
\defref{def:llperversite}, we get
\begin{eqnarray*}
 \nu^* D\nu_{*}\ov{p}(S)
 &=&
 D\nu_{*}\ov{p}(T)=\codim T-2-\nu_{*}\ov{p}(T)
 =
 \codim R-2-\ov{p}(R)\\
 &=&D\ov{p}(R)\leq D\ov{p}(S).
\end{eqnarray*}

For the second part of the statement, we still use \propref{prop:stratifieethomotopie} and we are reduced to
the proof of the inclusion $\nu(X_{\ov{p}})\subset \Sigma_{X^*}$. For that, we consider
$S\in \cS_{X}$ such that $\ov{t}(S) < \ov{p}(S)$ and we have to show 
$\nu(S)\subset \Sigma_{X^*}$. Let us suppose that this last inclusion is false. 
Then $\nu(S)$ is included in a regular stratum $T$ of $X^*$ and there exists a source $R\in\cS_{X}$ of $T$
such that $S\sim R$, see \propref{prop:source}.
Property (D) of $\ov{p}$ implies $D\ov{p}(S)\geq 0$ which is equivalent to
$\ov{p}(S)\leq \ov{t}(S)$. We are arriving to a contradiction, thus $\nu(S)$ is included in $\Sigma_{X^*}$
and $\nu$ induces a chain map,
$\nu_{*}\colon \gC^{\ov{p}}_{*}(X;G)\to \gC^{\nu_{*}\ov{p}}_{*}(X^*;G)$.
\end{proof}

\begin{theorem}\label{thm:invariancegenerale}
Let $(X,\ov{p})$  be a perverse CS set with a K-perversity.  
Then, the intrinsic aggregation $\nu \colon X \to X^*$
verifies the following properties.
\begin{enumerate}[(i)]
\item The map  $\nu$ 
 induces an isomorphism
$$H_{*}^{\ov{p}}(X;G) \cong  H_{*}^{\nu_{*}\ov{p}}(X^*;G).$$
\item Moreover, if $\ov p$ is a K$^*$-perversity, then the map  $\nu$ also induces an isomorphism
$$\gH_{*}^{\ov{p}}(X;G) \cong  \gH_{*}^{\nu_{*}\ov{p}}(X^*;G).$$
\end{enumerate}
\end{theorem}

This  result contains the topological invariance of  Goresky and MacPherson, 
\cite[Section 4.1]{GM1} and the more general version of  King 
\cite{MR800845}, (see also \cite[Theorems 5.5.1]{FriedmanBook}). Let us see that.

\begin{corollaire}\label{cor:kinginvariance}
Let $X$ be a CS set endowed with  a King positive perversity, $\ov{p}$.
The intrinsic aggregation   $\nu \colon X \to X^*$ induces the isomorphism
$H_{*}^{\ov{p}}(X;G) \cong H_{*}^{\ov{p}}(X^*;G)$.
\end{corollaire}

\begin{proof} 
From \remref{69}, the perversity $\ov p$ is a K-perversity.
Thus  \thmref{thm:invariancegenerale} gives the result if we prove  $\nu_{*}\ov{p}=\ov{p}$.
Let us notice that in this equality, the perversity $\ov{p}$ is a codimension-dependent  perversity
while $\nu_{*}\ov{p}$ is defined on the set of strata. 
However, we can consider $\ov{p}$ as a map $\ov{p}\colon \cS_{X}\to \N$ by
$\ov{p}(S)=\ov{p}(\codim S)$.
Thus this equality makes sense and we establish it by considering 
$T\in\cS^*$  and $S$  a source stratum of $T$. Then, we have,
$
\nu_* \ov p(T) = \ov p(S)  =\ov p(\codim S) = \ov p (\codim T)=\ov{p}(T).
$
\end{proof}

We deduce now from \thmref{thm:invariancegenerale} the behaviour of the two intersection homologies
 in the case of a refinement. 
As recalled in the introduction, refinements have been treated by G. Valette in the PL  framework  \cite{MR3175250}
and by G. Friedman in
\cite{MR2209151}. We answer this question here for CS sets. As we have to consider several stratifications 
on the same topological space, we denote by the pair $(X,\cS)$  the topological space $X$
 together with a set of strata $\cS$.
The corresponding $\ov{p}$-intersection homology with coefficient in $G$ is denoted $H_{*}^{\ov{p}}(X,\cS;G)$
or simply $H_{*}^{\ov{p}}(X,\cS)$ if there is no ambiguity.

\begin{definition}\label{def:refinement}
A CS set $(X,\cS)$  is a \emph{ refinement } of the
CS set $(X,\cT)$, if any stratum of $\cT$ is a union of strata of $\cS$.
Let us observe that this relation is equivalent to the fact
that the identity map, induces a stratified map $\kappa \colon (X,\cS) \to (X,\cT)$.
\end{definition}

In particular, any CS set is a refinement of its associated intrinsic CS set.

\begin{corollaire}\label{refin}
Let $\kappa\colon (X,\cS) \to (X,\cT)$ be a refinement between two CS sets. For any K-perversity $\ov p$ on $(X,\cT)$, the identity induces the isomorphism
\begin{equation}\label{uno}
H_{*}^{\kappa^*\ov{p}}(X,\cS;G) \cong  H_{*}^{\ov{p}}(X,\cT;G).
\end{equation}
Moreover, if $\ov p$ is a K$^*$-perversity  and 
there is no  codimension one stratum $S \in \cS$ becoming regular in $\cT$, then 
 the identity map induces an isomorphism
\begin{equation}\label{due}
\gH_{*}^{\kappa^*\ov{p}}(X,\cS;G) \cong  \gH_{*}^{\ov{p}}(X,\cT;G).
\end{equation}
\end{corollaire}

\begin{remarque}\label{rem:invariance}
In particular, for any perversity $\ov{p}$ on $X^*$, we have
$
H_{*}^{\nu^*\ov{p}}(X;G) \cong  H_{*}^{\ov{p}}(X^*;G)
$ and
$
\gH_{*}^{\nu^*\ov{p}}(X;G) \cong  \gH_{*}^{\ov{p}}(X^*;G)
$,
with the previous property on regular strata.
Let us also observe that the perversities $\ov{q}$ on $X$ that are obtained as pull-back's of a perversity
defined on $X^*$ are characterized by
$$S\sim S' \Rightarrow
\ov{q}(S)=\ov{q}(S').$$
\end{remarque}

\begin{proof}[Proof of \corref{refin}]
Both  CS sets have the same intrinsic space since it does not depend on the stratification. 
Thus, the two intrinsic aggregations
$\nu\colon (X,\cS) \to (X^*, \cS^*)$ and $\nu'\colon (X,\cT) \to (X^*, \cT^* )= (X^*, \cS^*)$
verify $\nu' \circ \kappa = \nu$.

Let  $S,Q \in \cS$. (Recall \defref{def:RelEquiv2} and the notation $S^f$ from \defref{def:applistratifiee}.) We have:
\begin{enumerate}[(i)]
\item $
S \sim Q \Longleftrightarrow  S^\nu \sim   Q^{\nu}   \Longleftrightarrow 
(  S^{\kappa})^{\nu'}  \sim (  Q^{\kappa})^{\nu'}  
\Longleftrightarrow 
  S^\kappa   \sim   Q^\kappa
$,
since two points belonging to the same stratum of a CS set are equivalent;
\item $
S \preceq Q \Longleftrightarrow S \subset  \ov Q \Longrightarrow  S^{\kappa} \subset \ov { Q^\kappa}
  \Longleftrightarrow   S^\kappa \preceq  Q^\kappa
$;
\item If $S$ is a source stratum of $U\in\cS^*$, then
$S^\kappa\in \cT$ is a source stratum of $U$. Let us see that.
We have $U=  S^\nu$ and $\dim S = \dim U$. This gives
$U = (  S^\kappa)^{\nu'}$. 
Moreover, since $\nu'$ and $\kappa$ are stratified maps,
we have
$
\codim U \leq \codim  S^\kappa \leq \codim S = \codim U
$. We deduce $\dim U = \dim  S^\kappa$. 
\end{enumerate}

\medskip
$\bullet$
\thmref{thm:invariancegenerale} gives the first assertion \eqref{uno} if we prove that
\begin{enumerate}[(a)]
\item the perversity $\kappa^*\ov p$ is a K-perversity,
\item and $\nu'_* \ov p = \nu_* \kappa^* \ov p$.
\end{enumerate}
Let us  verify the properties (A), (B) and (C) of \defref{def:llperversite} for $\kappa^*\ov{p}$ and (b).

(A) Let $S,R \in \cS$ be two strata with $S \sim R$ and $R$ regular.
If $ S^\kappa $ is regular, we have $\kappa^* \ov p (S) = \ov p(  S^\kappa) =0$. 
Suppose now that $S^\kappa $ is singular. Since the map $\kappa$ is a stratified map, 
then  $ R^\kappa$ is a regular stratum. 
We know that $ S^\kappa \sim  R^\kappa$. Thus,  Property (A) for the perversity ${\ov p}$ gives:
$
\kappa^* \ov p (S) = \ov p(  S^\kappa) \geq 0
$.

(B) Let $S,S' \in \cS$ where $S'$ is  singular  and a source stratum of  $U \in \cS^*$. 
We suppose $S \preceq S' \text{ and } S\sim S' $.
From previous properties (i), (ii), (iii), we deduce that
${S'}^\kappa $ is singular  and a source stratum of  $U$, with
$S^\kappa \preceq  {S'}^\kappa$   and $  S^\kappa \sim  {S'}^\kappa$. Thus, Property (B) for the perversity $\ov{p}$ gives
$\ov p (  {S'}^\kappa  )\leq \ov p (  S^\kappa  )$
and
$D\ov p (  {S'}^\kappa  )\leq D\ov p (  S^\kappa  )$
which is equivalent to
$$\kappa^*\ov p (  {S'}  )\leq \kappa^*\ov p (  S )
\text{ and }
\kappa^*D\ov p (  {S'}  )\leq \kappa^*D\ov p (  S  ).$$
It remains to prove that  $D  \kappa^* \ov p (S') \leq D   \kappa^* \ov p (S)$. 
Let us notice that $\dim S' =\dim U =\dim  {S'}^\kappa $. Since $\kappa$ is a stratified map, we have
\begin{eqnarray*}
D  \kappa^* \ov p (S')
&=&
\codim S' - 2 - \kappa^* \ov p (S') =
\codim   {S'}^\kappa   - 2 -  \ov p (  {S'}^\kappa ) \\
&=& D \ov p (  {S'}^\kappa  ) 
\leq  D \ov p (   {S}^\kappa  )  = \codim (  {S}^\kappa  ) -2 -\ov p (  {S}^\kappa  )
  \\
 &\leq &  \codim S -2 -\kappa^* \ov p (S) = D\kappa^*\ov p(S).
 \end{eqnarray*} 
 
(C) If $S,S' \in \cS$ are two source strata of a stratum $U \in \cS^*$, then 
$ S^\kappa,  {S'}^\kappa \in \cT$ are two source strata of $U$, cf. (iii). Thus using Property (C) for $\ov{p}$ gives 
$$
\kappa^* \ov p (S) = \ov p (  S^\kappa  ) = \ov p (  {S'}^\kappa ) = \kappa^* \ov p (S').
$$

(b) Let $U$ be a stratum of $\cS^*=\cT^*$. Consider $S\in\cS$ a source stratum of $U$. 
Since $S^{\kappa}\in\cT$ is a source stratum of $U$ (cf. (iii)), then we have
$$\nu_{*}\kappa^*\ov{p}(U)=
\kappa^*\ov{p}(S)=\ov{p}(S^{\kappa})=\nu'_{*}\ov{p}(U).
$$

\medskip 
$\bullet$
The proof of \eqref{due} is reduced to the verification of Property  (D) of \defref{def:llperversite}.
 Let $S,R \in \cS$  be two strata with $S \sim R$, $S$ singular and $R$ regular.
 
-- If $ S^\kappa $ is regular, we have $D\kappa^* \ov p (S) = \ov t  (S) - \ov p(  S^\kappa) =\ov t (S)$. 
Since $S$ is singular and $\bi S \kappa$ is regular, we have $\codim S\geq 2$ by hypothesis. 
Thus, $D\kappa^* \ov p (S) \geq 0$.

-- If $ S^\kappa $ is singular, since the map $\kappa$ is a stratified map, the stratum  
$ R^\kappa$ is regular  and $\codim  S^\kappa \leq \codim S$. 
Thus, Property (D) for the perversity $\ov{p}$ implies,
$
D\kappa^* \ov p (S) = \codim S - \ov p ( S^\kappa) -2
 \geq 
 \codim  S^\kappa - \ov p ( S^\kappa) -2
 =
  D \ov p ( S^\kappa) \geq 0
$.
\end{proof}

The second statement of \corref{refin} also retrieves a result of Greg Friedman, without any restriction on the strata
of codimension one.

\begin{corollaire}{(\cite{MR2209151})}\label{cor:greg}
Let $X$ be a CS set, endowed with  a King positive perversity, $\ov{p}$.
Then the tame $\ov{p}$-intersection homology is invariant under restratifications that fix the singular set $\Sigma$
of $X$.
\end{corollaire}

\begin{proof}[Proof of \corref{cor:greg}]
We only have to observe that $\ov{p}$ is a $K^*$-perversity in the sense of \defref{def:llperversite}.

$\bullet$  Conditions (A) and (D) are satisfied, since the singular set is fixed.

$\bullet$   Condition (B) is exactly (\ref{equa:king}).

$\bullet$   Condition (C) is a consequence of the fact that two source strata of one stratum have the same dimension.
\end{proof}

The proof of  \thmref{thm:invariancegenerale} uses  \thmref{thm:gregtransformationnaturelle}. 
For doing that, we need to make explicit the behavior of a conical neighborhood.

\begin{proposition}\label{prop:local}
Let $(X,\ov{p})$ be a perverse  CS set where $\ov{p}$ is a  $\K$-perversity.
  Consider a stratum $S \in \cS$ and  $(U,\varphi)$ a conical chart of a point $x\in S$. 
If $\nu\colon X\to X^*$  is the intrinsic aggregation, then the following implication is verified
 $$\nu_{*}\colon 
 H^{\ov{p}}_{*}(U\backslash S;G)
 \xrightarrow[]{\cong}
 H^{\nu_{*}\ov{p}}_{*}((U\backslash S)^*;G)
 \Longrightarrow
 \nu_{*}\colon 
 H^{\ov{p}}_{*}(U;G)
 \xrightarrow[]{\cong}
 H^{\nu_{*}\ov{p}}_{*}(U^*;G)
 .$$
Moreover, if  $\ov p$ is a K$^*$-perversity, we have also 
 $$\nu_{*}\colon 
 \gH^{\ov{p}}_{*}(U\backslash S;G)
 \xrightarrow[]{\cong}
 \gH^{\nu_{*}\ov{p}}_{*}((U\backslash S)^*;G)
 \Longrightarrow
 \nu_{*}\colon 
 \gH^{\ov{p}}_{*}(U;G)
 \xrightarrow[]{\cong}
 \gH^{\nu_{*}\ov{p}}_{*}(U^*;G)
 .$$
\end{proposition}

\begin{proof}
From \propref{prop:source2}, we know that the map  $\nu_{*}$ is well defined.
We prove the first statement.
Let us begin by studying the CS set structures.
Without loss of generality, we can suppose  
$U = \R^k \times \rc W$, where $W$ is a compact filtered space, $S \cap U= \R^k \times \{ \tw\}$
and $\tw$ is the apex of the cone $\rc W$.
Following \cite[Proof of Proposition 6]{MR800845}, 
there exists a homeomorphism of stratified spaces
\begin{equation}\label{equa:homeo}
h\colon ( \R^k \times \rc W)^*\xrightarrow[]{\cong} \R^m\times \rc L,
\end{equation}
where $L$ is a compact filtered space  (maybe empty)
and $m\geq k$. 
Moreover, the map $h$ verifies,
\begin{equation}\label{equa:hetcone}
h(\R^k\times \{\tw\} ) \subset \R^m\times \{\tv\} \text{ and }
h^{-1}(\R^m\times \{\tv\})=\R^k\times \rc A,
\end{equation}
where $\tv$ is the  apex of the cone  $\rc L$ and $\R^k\times A \times ]0,1[$
is a manifold which is also an $(m-k-1)$-dimensional homological sphere.
Using these notations, the hypothesis and the conclusion of the first statement become
\begin{equation}\label{equa:hyp}
h \colon 
 H^{\ov{p}}_{*}(\R^k  \times \rc W \backslash (\R^k  \times \{\tw\})) 
\xrightarrow[]{\cong}
 H^{\nu_*\ov{p}}_{*}(\R^m  \times \rc L \backslash h(\R^k  \times \{\tw\}))
\end{equation}
and
\begin{equation}\label{equa:conc}
h\colon 
 H^{\ov{p}}_{*}(\R^k  \times \rc W)
\xrightarrow[]{\cong}
 H^{\nu_*\ov{p}}_{*}(\R^m  \times \rc L).
\end{equation}
We write $s = \dim W$  and $t=\dim L$.  
The homeomorphism \eqref{equa:homeo} implies  $k+s=m+t$ and therefore $s\geq t$ since $m\geq k$.

$\bullet$ \emph{First case. We suppose $t=-1$}. The result is immediate if $s=t=-1$,
that is $W=L=\emptyset$. 
So, we can suppose  $s\geq 0$ and   $t=-1$, that is $L=\emptyset$ and $W\neq \emptyset$.
As $L=\emptyset$ and $A=W$ is a homological sphere,
we deduce from \eqref{equa:hyp} and \eqref{equa:homeo}
$$ H^{\ov{p}}_{*}(W)=H^{\ov{p}}_{*}(\R^k  \times \rc W \backslash (\R^k  \times \{\tw\})) \cong H_{*}(W)
=H_{0}(W)\oplus H_{\dim W}(W).
$$
Following (A), we have $\ov{p}(\tw)=\ov{p}(\R^k\times \{\tw\})\geq 0$.
Thus the value of the $\ov{p}$-intersection homology
of a cone, determined in \propref{prop:homologiecone}, reduces to
$$
H^{\ov{p}}_{k}(\rc W)
\cong
\left\{
\begin{array}{ll}
H^{\ov{p}}_{k}( W)& \hbox{if } 0=k< \dim W-\ov{p}( \{ \tw\}),\\
G& \hbox{if } 0= k \geq \dim W-\ov{p}( \{ \tw \}),\\
0& \hbox{if } 0\ne k \geq \dim W-\ov{p}( \{ \tw \}).
\end{array}
\right.
$$
In the two cases, $\dim W=0$ or $\dim W>0$, we get 
$H^{\ov{p}}_{*}(\rc W)=G$.
In conclusion, we have
$$H^{\ov{p}}_{*}(\R^k\times \rc W)
= H^{\ov{p}}_{*}(\rc W)  =G=H^{\nu_{*}\ov{p}}_{*}(\R^n\times \rc L),
$$
(cf.  \corref{cor:RfoisX})  where the  last  equality  comes fom $L=\emptyset$.

\medskip
\emph{So, for the rest of the proof, we can suppose  $t\geq 0$  and  $s\geq 0$,} that is 
$W\neq \emptyset$
and $L\neq \emptyset$.

\propref{prop:source} gives a source stratum   $S_{m}$ of $\R^m \times \{ \tv\}$ verifying: $S_m$ is singular, 
$\R^k\times \{\tw\} \preceq S_{m}$
and
$\R^k\times \{\tw\}\sim S_{m}$.
We also have $\nu_*\ov{p}(\R^m \times \{ \tv\} ) = \ov{p}(S_m)$. 
Since $\ov{p}$ is a K-perversity, we deduce
\begin{eqnarray} \label{equa:kifo}
\ov{p}(\R^k \times \{ \tw\}) - \nu_*\ov{p}(\R^m \times \{ \tv\}) &=& \\  \nonumber
\ov{p}(\R^k \times \{ \tw\})  - \ov{p}(S_m) 
\leq _{(B)_{\ov p}}
\codim (\R^k\times \{ \tw\}) - \codim S_m  
&=&\\  \nonumber
\codim (\R^k\times \{ \tw\}) - \codim (\R^m\times \{ \tv\}) 
= \dim W - \dim L &=& s-t.
\end{eqnarray}

We continue the proof of   (\ref{equa:conc}),
$h\colon 
 H^{\ov{p}}_{i}(\R^k  \times \rc W)
\xrightarrow[]{\cong}
 H^{\nu_*\ov{p}}_{i}(\R^m  \times \rc L)$, by distinguishing several steps
 following the values of $i$.

\smallskip
$\bullet$ \emph{Second case. We suppose $0\neq i \geq s - \ov{p}(\R^k \times \{ \tw\} )$.}\\
  \corref{cor:RfoisX} and   \propref{prop:homologiecone} imply
$H^{\ov{p}}_{i}(\R^k \times \rc W)=0$  and we need to prove $ H^{\nu_*\ov{p}}_{i}(\R^m \times \rc L)=0$. 
We apply  \eqref{equa:kifo} and  obtain
$$i\geq s - \ov{p}(\R^k \times \{ \tw\}) \geq 
t - \nu_*\ov{p}(\R^m \times \{ \tv\} )= t-\nu_{*}\ov{p}(\{\tv\}).$$
 \corref{cor:RfoisX} and   \propref{prop:homologiecone} 
give
$ H^{\nu_*\ov{p}}_{i}(\R^m \times \rc L) =0$.

\smallskip
$\bullet$ \emph{Third case. We suppose $0 = i \geq s - \ov{p}(\R^k \times \{ \tw\})$}.\\
  \corref{cor:RfoisX} and   \propref{prop:homologiecone} imply
$ H^{\ov{p}}_{0}(\R^k \times \rc W;G)=G$. 
This group is generated by any point of the regular stratum of $\R^k \times \rc W$. 
Let us notice that $\nu$ sends this point on a  regular stratum of $\R^m \times \rc L$.
We  compute  $ H^{\nu_*\ov{p}}_{0}(\R^m \times \rc L;G)$.  
We apply (\ref{equa:kifo}) and we obtain
 $$s\leq  \ov{p}(\R^k \times \{ \tw\})\leq \nu_*\ov{p}(\R^m \times \{ \tv\} ) +s-t,$$
hence $t \leq \nu_*\ov{p}(\R^m \times \{ \tv\} )$. Thus,
 \corref{cor:RfoisX} and \propref{prop:homologiecone} imply 
 $ H^{\nu_*\ov{p}}_{0}(\R^m \times \rc L;G)=G$,  generated by any point of the regular stratum of $\R^m\times \rc L$. 
 This gives \eqref{equa:conc} since $h$ sends regular points to regular points.

\smallskip
$\bullet$ \emph{Fourth case. We suppose $i < s - \ov{p}(\R^k \times \{ \tw\})$}. \\
We have the following isomorphisms,
\begin{equation}\label{equa:cas3debut}
 H^{\nu_{*}\ov{p}}_{i} (\R^m\times \rc L\backslash h(\R^k\times \{\tw\}))
\cong_{(1)}
  H^{\ov{p}}_{i}((\R^k\times\rc W)\backslash (\R^k\times \{\tw\})) 
\cong_{(2)} 
  H^{\ov{p}}_{i}(\R^k\times \rc W),
\end{equation}
where $\cong_{(1)}$ is the hypothesis (\ref{equa:hyp}) and $\cong_{(2)}$ is   given by \corref{cor:RfoisX} and \propref{prop:homologiecone}.  
The commutative diagram
$$
\xymatrix{
\R^m\times \rc L\backslash h(\R^k\times \{\tw\}) 
\ar@{^(->}[d]^\iota&
\ar[l]_-h (\R^k\times\rc W)\backslash (\R^k\times \{\tw\}) \ar@{^(->}[d]^\iota\\
\R^m  \times \rc L 
&\ar[l]_-h\R^k\times \rc W
}
$$
gives (cf. \eqref{equa:cas3debut}) that \eqref{equa:conc}
 becomes
 \begin{equation}\label{equa:concTris}
\iota_* \colon H_i^{\nu_* \ov p} (\R^m\times \rc L\backslash h(\R^k\times \{\tw\})) \xrightarrow \cong
H_i^{\nu_* \ov p} ( \R^m  \times \rc L).
\end{equation} 
 Consider $h(\R^k\times\{\tw\})=B\times \{\tv\}\subset \R^m\times \{\tv\}$ with $B$  a closed subset.
Using the excision property (cf. \corref{prop:Excisionhomologie}) we get the following isomorphisms
\begin{align}
 H^{\nu_{*}\ov{p}}_{i}(\R^m\times \rc L\backslash h(\R^k\times \{\tw\}),\R^m\times \rc L\backslash \R^m\times\{\tv\})
&\cong&\nonumber \\
 H^{\nu_{*}\ov{p}}_{i}((\R^m\times \rc L )\backslash (B\times\{\tv\}) ,\R^m\times (\rc L\backslash \{\tv\}))
&\cong_{(1)}&\nonumber \\
 H^{\nu_{*}\ov{p}}_{i} ((\R^m\backslash B)\times \rc L, (\R^m\backslash B)\times (\rc L\backslash \{\tv\}))
&\cong&\nonumber \\
 H^{\nu_{*}\ov{p}}_{i} ((\R^m\backslash B)\times (\rc L,\rc L\backslash \{\tv\}))
&\cong_{(2)}&\nonumber \\
 H^{\nu_{*}\ov{p}}_{i} (\R^{k+1}\times A\times (\rc L,\rc L\backslash \{\tv\}))
&\cong_{(3)}&\nonumber \\
 H^{\nu_{*}\ov{p}}_{i}(\rc L,\rc L\backslash \{\tv\})\oplus 
 H^{\nu_{*} \ov{p}}_{i-m+1+k} (\rc L,\rc L\backslash \{\tv\}), \label{equa:tropdeL}
\end{align}
where $\cong_{(1)}$ is the excision of $B\times (\rc L\backslash \{\tv\})$,
$\cong_{(2)}$ comes from (\ref{equa:hetcone}) and $\cong_{(3)}$ from the relative version of
\corref{cor:SfoisX}.
Let us notice that this formula applies for any degree~$i$. In particular, the canonical projection 
$
 \pr_2 \colon \R^m\times \rc L \to \rc L
 $
induces the epimorphism
\begin{equation} \label{eq:EpideL}
 H^{\nu_{*}\ov{p}}_{i+1}(\R^m\times \rc L\backslash h(\R^k\times \{\tw\}),\R^m\times \rc L\backslash \R^m\times\{\tv\})
\to
 H^{\nu_{*}\ov{p}}_{i+1}(\rc L,\rc L\backslash \{\tv\}).
 \end{equation}
Since $i < s - \ov{p}(\R^k \times \{ \tw\})$, we have
$$i-m+1+k<s - \ov{p}(\R^k \times \{ \tw\})-m+1+k=t-\ov{p}(\R^k \times \{ \tw\})+1.$$
 \propref{prop:source} gives a source stratum   $S_{m}$ of $\R^m \times \{ \tv\}$ verifying 
$\R^k\times \{\tw\} \preceq S_{m}$
and
$\R^k\times \{\tw\}\sim S_{m}$.
We also have $\nu_*\ov{p}(\R^m \times \{ \tv\} ) = \ov{p}(S_m)$. 
Since  $\ov{p}$ is a  K-perversity, then
\begin{equation*}\label{equa:equa}
0\leq \ov{p}(\R^k \times \{ \tw\})  - \ov{p}(S_m) =\ov{p}(\R^k \times \{ \tw\}) - \nu_*\ov{p}(\R^m \times \{ \tv\}).
\end{equation*}
We get $i-m+1+k\leq t- \nu_*\ov{p}(\R^m \times \{ \tv\})$.
Since $L\neq\emptyset$, using  \corref{cor:homologieconerel},  we conclude that 
the second term of the direct sum  (\ref{equa:tropdeL}) vanishes.
We have proven that the canonical projection $\pr_2$ induces the isomorphism
\begin{equation}\label{eq:IsodeL}
H^{\nu_{*}\ov{p}}_{i}(\R^m\times \rc L\backslash h(\R^k\times \{\tw\}),\R^m\times \rc L\backslash \R^m\times\{\tv\})
\xrightarrow {\cong}
 H^{\nu_{*}\ov{p}}_{i}(\rc L,\rc L\backslash \{\tv\}).
 \end{equation}
 Set $A=(\R^m \times \rc L) \menos h(\R^k \times \{ \tw\})$
 and $B=(\R^m \times \rc L) \menos (\R^m \times \{ \tv\})$.
The long exact sequences associated to the pairs
$(A,B)$
and 
$
(\rc L, \rc L \menos \{ \tv\})
$ 
are related as follows
$$
\xymatrix@C=3,5mm{
 H_{i+1}^{\nu_*\ov p}(A,B) \ar[r]  \ar[d]^{P_{A,B} }  & H_i^{\nu_*\ov p} (B) \ar[r]  \ar[d]^{P_B} & H_i^{\nu_*\ov p} (A) \ar[r]  \ar[d]^{P_A}  
& H_{i}^{\nu_*\ov p}(A,B) \ar[r]  \ar[d]^{P_{A,B}}& H_{i-1}^{\nu_*\ov p}(B) \ar[d]^{P_B} \\
H^{\nu_*\ov p}_{i+1}(\rc L,\rc L\backslash \{\tv\})
\ar[r]&
H^{\nu_*\ov p}_i(\rc L\backslash \{\tv\}) \ar[r] &H^{\nu_*\ov p}_i(\rc L) \ar[r] &H^{\nu_*\ov p}_i(\rc L,\rc L\backslash \{\tv\})
\ar[r]&H^{\nu_*\ov p}_{i-1}(\rc L\backslash \{\tv\}) 
}
$$
where the vertical maps are induced by the canonical projection  $\pr_2 $.
  \corref{cor:RfoisX} gives us that the morphisms $P_B$ are isomorphisms. Keeping in mind \eqref{eq:EpideL} and  \eqref{eq:IsodeL}, we  apply the Five Lemma and conclude that $P_A$ is an isomorphism. Thus, the map 
$ \pr_{2,*} \colon H_i^{\nu_*\ov p} (\R^m\times \rc L\backslash \R^m\times\{\tv\}) \to H_i^{\nu_*\ov p} ( \rc L) $
 is an isomorphism.
Finally,   \corref{cor:RfoisX} gives \eqref{equa:concTris}.
 
\smallskip
This implies the first statement.

\medskip
The proof of the second statement follows the same process, except for the intersection homology 
of a cone  in degree~0.
We distinguish the same cases.

$\bullet$ \emph{First case}.  We can suppose $s\geq 0$ and $t=-1$, that is:
$L=\emptyset$ and $A=W\neq\emptyset$ is a homological sphere. From \eqref{equa:hyp}, we deduce
$\gH_{*}^{\ov{p}}(W)=H_{0}(W)\oplus H_{\dim W}(W)$.
Finally, the hypotheses (A) and (D) imply
$$1\leq \dim W-\ov{p}(\{\tw\})\leq \dim W.$$
Recall from \propref{prop:homologiecone},
$$\gH^{\ov{p}}_{k}(\rc W)
\cong
\left\{
\begin{array}{ll}
\gH^{\ov{p}}_{k}( W)& \hbox{if } k< \dim W-\ov{p}( \{ \tw \}),\\
0 & \hbox{if }  k \geq \dim W-\ov{p}( \{ \tw \}).
\end{array}\right.$$
We consider the first line. The previous observations show that the condition $k=\dim W$ cannot occur and that
the only possibillity is $k=0$. Therefore, we get, as required,
$$ \gH_{*}^{\ov{p}}(\R^k \times \rc W)  =  \gH_{*}^{\ov{p}}(\rc W) =H_{0}(W)=G = H_0^{\nu_* \ov p}(\R^m \times \rc L)  = H_*^{\nu_* \ov p}(\R^m \times \rc L).
$$

$\bullet$ \emph{Second case}.  This is the same proof.

$\bullet$ \emph{Third case}. Since $0 = i \geq s - \ov{p}(\R^k \times \{ \tw\} )$ and  
$ 0=i \geq t - \nu_*\ov{p}(\R^m \times \{ \tv\} )$,
 we get   $  \gH^{\ov p}_{0}(\rc W) =0 =   \gH^{\nu_*\ov p}_{0}(\rc L)$ from \propref{prop:homologiecone}.

$\bullet$ \emph{Fourth case}. This is the same proof.
\end{proof}

\begin{proof} [Proof of \thmref{thm:invariancegenerale}]
The proof is reduced to the verification of the hypotheses of \thmref{thm:gregtransformationnaturelle}, where
$\Phi_{U}\colon H^{\ov{p}}_{*}(U)\to H^{\nu_{*}\ov{p}}_{*}(U^*)$
(resp. $\Phi_{U}\colon \gH^{\ov{p}}_{*}(U)\to \gH^{\nu_{*}\ov{p}}_{*}(U^*)$)
is the natural transformation induced by the intrinsic aggregation $\nu^{U}\colon U\to U^*$,
see \propref{prop:source}(c) and
\propref{prop:source2}.
First, conditions (ii) and  (iv) of \thmref{thm:gregtransformationnaturelle} are direct. 
Condition (i) comes from \propref{prop:MVhomologie}
and condition (iii) is exactly  \propref{prop:local}.
\end{proof}

\begin{remarque}\label{rem:yes}
Here, we discuss the necessity of the conditions imposed in \defref{def:llperversite}. 

$\bullet$ First, let us observe that condition (C) is the key point for defining $\nu_{*} \ov{p}$ from $\ov{p}$. 
This is the keystone of this topological invariance.

$\bullet$ We present now an example illustrating the necessity of conditions (A) and (D). 
We consider the cone on a sphere,
$X=\rc S^n$, with one singular stratum $S=\{\tv\}$. We have:\\
--- $S\sim R$ with $R$ a regular stratum,\\
--- $X^*=E$, the open unit disk of dimension $n+1$.

We first choose the perversity $\ov{p}$ defined by $\ov{p}(S)=-1$, which implies
$D\ov{p}(S)=n$ and $\nu_{*}\ov{p}=0$. Therefore, \emph{(A) is not verified but (D) is.} Classical computations show that
the conclusion of \thmref{thm:invariancegenerale} is not true in this case:
\begin{eqnarray*}
\gH_{*}^{\ov{p}}(\rc S^n)
&=&
H_{*}^{\ov{p}}(\rc S^n)=H_{0}(S^n)\oplus H_{n}(S^n)=G\oplus G,\\
\gH_{*}^{\nu_{*}\ov{p}}(E)
&=&
H_{*}^{\nu_{*}\ov{p}}(E)=H_{*}(E)=G.
\end{eqnarray*}

Next, we set $\ov{p}(S)=n$, which implies $D\ov{p}(S)=-1$ and $\nu_{*}\ov{p}=0$. 
Therefore, \emph{(D) is not verified but (A) is.}
Here, we have
$\gH_{*}^{\ov{p}}(\rc S^n)=0$ and $\gH_{*}^{\nu_{*}\ov{p}}(E)=G$,
which also is in contradiction with the conclusion of \thmref{thm:invariancegenerale}.

$\bullet$ We develop an example, due to H. King (\cite{MR642001}), that \emph{proves the necessity of (B).}
We set $L=S^1\times S^1$ and $X=\rc (\Sigma L)$. This space has three singular strata, $S$, $S_{1}$ and $S_{2}$, with
$\dim S=0$, $\dim S_{1}=\dim S_{2}=1$, $S\preceq S_{1}$, $S\preceq S_{2}$ and $S\sim S_{1}\sim S_{2}$.
The intrinsic aggregation is $X^*=\R\times \rc L$, with one singular stratum, $T=S\sqcup S_{1}\sqcup S_{2}$, 
$\dim T=1$. We define a perversity on $X$ by $\ov{p}(S)=a$, $\ov{p}(S_{1})=\ov{p}(S_{2})=b$. (Let us notice
that this equality is condition (C) which allows the definition of $\nu_{*}\ov{p}$ on $X^*$.) By definition 
(see \propref{prop:source2}), we have $\nu_{*}\ov{p}(T)=b$. Classical computations give:
\begin{eqnarray*}
\gH_{*}^{\nu_{*}\ov{p}}(X^*)
&=&
\gH_{*}^{\nu_{*}\ov{p}}(\rc L)=H_{\leq b}(L),\\
\gH_{*}^{\ov{p}}(X)
&=&
\gH_{*}^{\ov{p}}(\rc \Sigma L)=\gH^{\ov{p}}_{\leq a}(\Sigma L)\\
&=&
\left\{
\begin{array}{ll}
H_{\leq a}(L) & \text{if } a\leq b,\\
H_{\leq b}(L) & \text{if } a=b+1,\\
H_{\leq b}(L) \oplus H_{b+1}(L)\oplus\dots\oplus H_{a-1}(L)& \text{if } a\geq b+2.
\end{array}\right.
\end{eqnarray*}
We deduce $\gH_{*}^{\nu_{*}\ov{p}}(X^*)=\gH_{*}^{\ov{p}}(X)$ if, and only if,
$a=b$ or $a=b+1$; i.e., 
$$\gH_{*}^{\nu_{*}\ov{p}}(X^*)=\gH_{*}^{\ov{p}}(X)
\Longleftrightarrow
\ov{p}(S)=\ov{p}(S_{i}) \text{ or } \ov{p}(S)=\ov{p}(S_{i})+1.$$
In conclusion, $\gH_{*}^{\nu_{*}\ov{p}}(X^*)=\gH_{*}^{\ov{p}}(X)$ if, and only if, (B) is verified.
\end{remarque}

\providecommand{\bysame}{\leavevmode\hbox to3em{\hrulefill}\thinspace}
\providecommand{\MR}{\relax\ifhmode\unskip\space\fi MR }
\providecommand{\MRhref}[2]{%
  \href{http://www.ams.org/mathscinet-getitem?mr=#1}{#2}
}
\providecommand{\href}[2]{#2}

\end{document}